\documentclass[10pt,a4paper]{article}
\usepackage{latexsym}
\usepackage{amsthm,amsmath,amssymb}
\usepackage{fullpage}
\usepackage{float}
\usepackage{amsfonts}
\usepackage{graphicx}
\usepackage{amsmath}
\usepackage{color}

\newfloat{figure}{H}{lof}
\newfloat{table}{H}{lot}  
\floatname{figure}{\figurename}  
\floatname{table}{\tablename}
\newtheorem{theorem}{Theorem}[section]

\newtheorem{lemma}[theorem]{Lemma}

\newtheorem{remark}[theorem]{Remark}

\newcommand{\gr}[1]{{%\color{green}
#1}}

\begin{document}

\title{Dynamics of a kinetic model describing protein exchanges in a cell population}

\maketitle

\begin{center}
{\large\bf Pierre Magal \footnote{University of Bordeaux, IMB, UMR CNRS 5251,
33076 Bordeaux, France. E-mail: pierre.magal@u-bordeaux.fr} and Ga\"el Raoul \footnote{CMAP, Ecole Polytechnique, CNRS, Université Paris-Saclay, Route de Saclay, 91128 Palaiseau cedex, France.
E-mail: raoul@cmap.polytechnique.fr}}\\
[2ex]

\end{center}

\begin{abstract} 
We consider a cell population structured by a positive real number, describing the number of P-glycoproteins carried by the cell. We are interested in the effect of those proteins on the growth of the population: those proteins are indeed involve in the resistance of cancer cells to chemotherapy drugs. To describe this dynamics, we introduce a kinetic model. We then introduce a rigorous hydrodynamic limit, showing that if the exchanges are frequent, then the dynamics of the model can be described by a system of two coupled differential equations. Finally, we also show that the kinetic model converges to a unique limit in large times. The main idea of this analysis is to use  Wasserstein distance estimates to describe the effect of the kinetic operator, combined to more classical estimates on the macroscopic quantities.
\end{abstract}

\section{Introduction}
\label{sec:intro}
In this study, we are interested a population of cells structured by a trait $x\in\mathbb R_+$, which measures the quantity of P-glycoprotein (P-gp) carried by the cell. We assume that cells create pairs with a given probability, and that those pairs proceed to an exchange, after which the links between the two cells disappear. During the exchange phase, each cell gives a portion of its protein to the other. We want to describe the dynamics of the density of individuals along the trait space, driven by those exchanges. The P-gp are membrane proteins, which play an important role in tumours. The cells can exchange their P-gp 
%, either through microparticles released in the medium, or thanks to direct connections between a pair of cells, 
through nanotubes \cite{Levchenko,Pasquier,Pasquier2}, and experiments show that those exchanges have a significant effect on the number of P-gp that the cells carry. The goal of this article is to study the effect of the exchanges on the number of P-gp that are carried by cells. %Actually, the cells also share their P-gp through microparticles released in the medium, but we have chosen to focus the present study on direct exchanges. 

In this study, we assume that all the cells are genetically identical cells. The trait corresponds to the quantity of P-gp at the surface of the cells. This quantity is measured by using the fluorescence of membrane-tagged  antibody. We assume that the trait of new born cells are drawn from a given distribution independent from the trait of the parent (we discuss possible generalisations of this model in the discussion section). In other words, we assume that there is no heritability of the trait. We will however assume that the trait of an individual has an effect on its reproduction rate. This is indeed the case when some chemotherapy drugs are present in the environment of the cells: the P-gp are membrane proteins that can pump several chemicals out of the cells, and which are in particular able to pump cytotoxic drugs out of the cells \cite{Pastan}. The P-gp thus play an important role in the emergence of chemotherapy resistance in tumour cell populations: cancer cells which carry a large number of P-gp are less susceptible to the drugs.

Notice that the exchange phenomenon that we will model and analyse here is related to several of other biological phenomena (see the discussion section), and we believe that the development of new mathematical methods for these phenomena will enable us to better  understand their effect.

\medskip

This article is structured as follows: in Section~\ref{sec:modelling}, we derive the model \eqref{eq:model-rep} from biological phenomena. We detail in particular the derivation of the exchange operator. Then, in Section~\ref{sec:results}, we state the two results of this study: a hydrodynamic limit of our kinetic model, and the long time convergence of the solutions of the kinetic model. We also recall the definition of Wasserstein distances, and some properties of those metric that we will use throughout this study. In section~\ref{sec:W2}, we the first result of our study, that is Theorem~\ref{thm:hydro}. The proof is based on the contracting property of the exchange operator for the $W_2-$Wasserstein distance. Finally, Section~\ref{sec:proof-thm2} is devoted to the second main result of this study, that is Theorem~\ref{thm:hypo}. The proof of this second result is based on the  $W_1-$Wasserstein distance.

\section{Modelling}
\label{sec:modelling}

\subsection{Derivation of the exchange operator}
\label{subsec:modelling-exchanges}

The model we consider here is based on  the assumption that when two cells (the cell $1$ and the cell $2$) interact, no information is exchanged: the cell $1$ does not know how many P-gp the cell $2$ carries. The number of P-gp the cell $1$ sends to the cell $2$ then only depends on its own number of P-gp. If we denote by $x_i$ (resp $x_i'$) the number of P-gp the cell $i$ contains before (resp. after) the exchange, and $X_i$ the number of P-gp the cell $i$ sends to the other cell, we have the relation:
\begin{equation}\label{xX}
x_1'=x_1-X_1+X_2.
\end{equation}
We have assumed that the number of pumps the cell $1$ sends to the cell $2$ only depends on $x_1$, so that the density of probability $x \rightarrow \mathcal B(x,x_1)$ of $X_1$ only depends on $x_1$. Notice that to be consistent biologically, a cell cannot give a negative number of pumps, or more pumps that it originally had, that is:
\begin{equation}\label{condbio1}
\forall x_1\in\mathbb R_+,\quad \textrm{supp }\mathcal B(\cdot,x_1)\subset [0,x_1].
\end{equation}
Note that we will be more precise on the dependency of $\mathcal B$ in $x_1$ later on. %\textcolor{red}{ Tu as besoin de quoi comme regularite $ x_1 \rightarrow \mathcal B(x,x_1)$? } 
We can relate the law of $x_1'$, that we denote $K(x,x_1,x_2)\,dx$, to the probability laws of $X_1$ and $X_2$, thanks to \eqref{xX}:
\begin{equation*}
\int_{0}^xK(x,x_1,x_2)\,dx=\mathbb P\left(x_1'\leq x|x_1,\,x_2\right)=
\int_0^\infty \int_0^{x-x_1+y}\,\mathcal B(z,x_2)dz\,\mathcal B(y,x_1)dy
%\int_0^\infty \int_0^{x-x_1+y}\,d\left(\mathcal B(\cdot,x_2)\right)(z)\,d\left(\mathcal B(\cdot,x_1)\right)(y),
\end{equation*}
that is, thanks to a derivation in $x$:
\begin{equation*}
K(x,x_1,x_2)=\int_0^\infty \mathcal B(w,x_1)\mathcal B(x-x_1+w,x_2)\,dw.
\end{equation*}
The above integral is well defined as a convolution. Moreover 
\begin{equation*}
\begin{array}{ll}
\int_0^\infty  K(x,x_1,x_2)dx&=\int_0^\infty \mathcal B(w,x_1) \int_0^\infty \mathcal B(x-x_1+w,x_2) \,dx\,dw\\
&=\int_0^{x_1} \mathcal B(w,x_1) \int_0^\infty \mathcal B(x-x_1+w,x_2) \,dx\,dw\\
&=\int_0^{x_1} \mathcal B(w,x_1) \int_0^\infty \mathcal B(\widehat{x},x_2) \,d\widehat{x}\,dw,
\end{array}
\end{equation*}
therefore
\begin{equation*}
\int_0^\infty  K(x,x_1,x_2)dx=1.
\end{equation*}

Since the problem is symmetric, the law of $x_2'$ is $K(x,x_2,x_1)$, and the collision operator can be written, for $u\in \mathcal P_2(\mathbb R)$:
\[\mathcal K(u)(x):=\int\int  \frac 12\left(K(x,x_1,x_2)+K(x,x_2,x_1)\right)u(x_1)u(x_2)\,dx_1\,dx_2-u(x),\]
which can also be written, thanks to a change of variable $(\tilde x_1,\tilde x_2)\to(x_2,x_1)$:
\[\mathcal K(u)(x):=\int\int  K(x,x_1,x_2)u(x_1)u(x_2)\,dx_1\,dx_2-u(x).\]

If we assume that the law of the number of P-gp the cell $1$ (resp. $2$) sends to the cell $2$ (resp. $1$) is proportional to the number of P-gp it originally contained, that is
\begin{equation}\label{Ass-B}
\mathcal B(y,x)=\frac 1xB\left(\frac yx\right),
\end{equation}
where $B\in \mathcal P(\mathbb R)$ is a probability density such that $\textrm{supp }B\subset[0,1]$. %This is a strong assumption, and it would be interesting to see wether it could be relaxed. We will however restrict the present study to the situations where \eqref{Ass-B} is satisfied. 
$K$ is then given by
\begin{equation}\label{def:K2}
K(x,x_1,x_2)=\frac 1{x_1x_2}\int_0^\infty\int_0^\infty\delta_{x=x_1-y_1+y_2} B\left(\frac {y_1}{x_1}\right)B\left(\frac {y_2}{x_2}\right)\,dy_1\,dy_2,
\end{equation}
 that is
 \begin{equation}\label{def:K3}
\int_0^\infty K(x,x_1,x_2)h(x)\,dx=\frac 1{x_1x_2}\int_0^\infty\int_0^\infty h(x_1-y_1+y_2) B\left(\frac {y_1}{x_1}\right)B\left(\frac {y_2}{x_2}\right)\,dy_1\,dy_2,
\end{equation}
for any test function $h\in L^\infty(\mathbb R_+)$. 
%which can be understood as
%\begin{equation*}
%\begin{array}{l}
%K(x,x_1,x_2)=\frac 1{x_1x_2}\int_0^\infty B\left(\frac {y_1}{x_1}\right)B\left(\frac {x-x_1+y_1}{x_2}\right)\,dy_1=\frac 1{x_1x_2}\int_0^\infty B\left(\frac {y_2-x+x_1}{x_1}\right)B\left(\frac {y_2}{x_2}\right)\,dy_2.
%\end{array}
%\end{equation*}
%Since $B$ is density of probability over the segment $[0,1]$, we must have
%\begin{equation}\label{Ass1}
%\textrm{supp }B \subset [0,1] \text{ and }\int_{(0,1)}B(x)\,dx=1.
%\end{equation}
In our study, we will assume that 
\begin{equation}\label{Ass1}
\int_{(0,1)}B(x)\,dx=1.
\end{equation}
%which is not a strong biological assumption:  $\int_{(0,1)}B(x)\,dx=0$ corresponds to a situation where a cell proceeding to an exchange either does not send any P-gp, or sends all the P-gp it contains. The assumption~\ref{Ass1} is satisfied in any other situation.

\subsection{Derivation of the model}
\label{subsec:modelling-model}

We consider a well-mixed population of cells, and assume that all the cells are genetically identical. Although genetically identical, the cell can differ by the number of P-gp they carry. The population is thus structured by a trait $x\in\mathbb R_+$, the number of P-gp that are present on the cell membrane. Let $n=n(t,x)$ represent the density of the population at time $t\geq 0$, along the trait $x\geq 0$.

We assume that the P-gp are only produced at birth, in a quantity that is drawn from a given distribution $n_b\in\mathcal P_2(\mathbb R)$. We assume that there is no heritability of this trait (the Heritability Index is $0$), this distribution is then independent from the trait of the parent. We assume however that the birth rate of a cell depends on its trait $x$: as described in the introduction, the P-gp are membrane proteins that pump chemotactic drugs out of the cell. If some drugs are present in the cell culture, the fitness of an individual will directly depend on the number of pumps it carries. We assume that the fitness of an individual is given by $r+\alpha(x)$, where  $\alpha\in W^{1,\infty}(\mathbb R_+)$. The birth is then
\[\left(\int \left(r+\alpha(y)\right) n(t,y)\,dy\right) n_b(\cdot)=\left(r+\int \alpha(y)\frac{n(t,y)}{N(t)}\,dy\right)N(t)n_b(\cdot),\]
where 
\begin{equation}\label{def:N}
N(t):= \int n(t,x)\,dx.
\end{equation}
For the death term, we assume that the death rate $\beta N(t)$ of the cells does not depend on their trait $x$, and is proportional to the total population. This assumption leads to the classical logistic regulation model, with the following death term:
\[-\beta N(t) n(t,\cdot).\]
During its life, each cell will proceed to exchanges at a rate $\gamma>0$ independent of the size of the population (we assume that finding exchange partners is not a limiting factor). Moreover, we assume that the traits have no influence on the selection of the exchange partner, which is chosen uniformly among the population. Considering the exchange operator described in Subsection~\ref{subsec:modelling-exchanges}, the effect of the exchanges can then be represented as follows:
\[\gamma\left(\frac{1}{N(t)}\int\int K(\cdot,x_1,x_2)n(t,x_1) n(t,x_2)\,dx_1\,dx_2- n(t,\cdot)\right).
\]

Bringing all those terms together, we obtain the following model:
%\subsection{The model}
%\label{subsec:model}
%
%Let $n=n(t,x)$ represent the density of the population at time $t\geq 0$, along the trait $x\geq 0$. We assume that an individual of trait $x$ reproduces with the rate $r+\alpha(x)$, giving birth to an offspring whose trait is drawn from a given distribution $n_b$, and dies with a rate $\beta \int n(t,y)\,dy$ proportional to the total population size (this is a the so-called logistic regulation, a classical assumption in population dynamics). Finally, each individual proceeds to an exchange with another individual (chosen independently from the traits of the two individuals involved), with a rate $\gamma$. The model then writes:
\begin{eqnarray}
\partial_t n(t,x)&=& \left(r+\int \alpha (y) \frac{n(t,y)}{N(t)}\right)\left(\int n(t,y)\,dy\right)n_b(x)-\beta N(t) n(t,x)\nonumber\\
&&+\gamma\left(\frac1{N(t)}\int\int K(x,x_1,x_2)n(t,x_1) n(t,x_2)\,dx_1\,dx_2- n(t,x)\right),\label{eq:model-rep}
\end{eqnarray}
where $r,\,\beta,\,\gamma>0$, $\alpha\in W^{1,\infty}(\mathbb R_+)$, $K$ is defined by \eqref{def:K2}, and the population size $N(t)$  is given by \eqref{def:N}.

\smallskip

\gr{In this study, we have chosen to focus our attention on the dynamics of solutions, rather than their existence and uniqueness, for which we refer to \cite{Bassetti2011}. The kinetic exchange term indeed induces some difficulties to show the existence of solutions. If $\textrm{supp }B\subset[\delta,1]$, for some $\delta>0$, and $n_b$ is smooth, then the proof of Lemma 3.1 from \cite{Bisi2009} can be reproduced to show the existence of solutions of \eqref{eq:model-rep}. In the manuscript, we will formulate our result for any solution $n=n(t,x)\in L^\infty_{loc}(\mathbb R_+,L^\infty(\mathbb R_+))$ of \eqref{eq:model-rep}.}

\section{Main results and comments}
\label{sec:results}

\subsection{Main results}
\label{subsec:main-results}

Our analysis will be based on Wasserstein distances, and more precisely, the $W_p-$Wasserstein distances (for $p=1$ or $p=2$). We refer to \cite{Villani,Carrillo-rev} for a description of those distances. Let us recall the properties of these metrics . The distance $W_p$ is defined on the set $\mathbb R_+$ of measures with a finite $p-$moment:
\[\mathcal P_p(\mathbb R_+):=\left\{\mu\geq 0\textrm{ a probability measure over }\mathbb R_+,\textrm{ such that }\int x^p\,d\mu(x)<\infty\right\}.\]
for two such probability measures $\mu,\,\nu\in \mathcal P_p(\mathbb R)$, we define
\begin{equation}\label{def:Wasserstein}
W_p(\mu,\nu):=\left(\sup_{\pi\in\Pi(\mu,\nu)}\int \int (x-y)^p\,d\pi(x,y)\right)^{\frac 1 p},
\end{equation}
where $\Pi_(\mu,\nu)$ is the set on measures on $\mathbb R_+^2$ with marginals $\mu$ and $\nu$, that is, for any measurable set $\omega\subset\mathbb R_+$,
\[\mu(\omega)=\pi(\omega\times \mathbb R_+),\quad \nu(\omega)=\pi(\mathbb R_+\times\omega).\]
%We will use on several occasions the Kantorovich dual formula for the $p-$Wasserstein distance, that we recall here. Let $p\in\{1,2\}$, and 
For $\mu,\,\nu\in \mathcal P_p(\mathbb R)$, the Kantorovich formula states that
\begin{equation}\label{def:Wasserstein-dual}
W_p(\mu,\nu)^p=\sup_{(\varphi,\psi)\in \Phi_p}\left(\int \varphi(x)d\mu(x)+\int\psi(X)\,d\nu(X)\right),
\end{equation}
where the suppremum is taken over all functions 
\begin{equation}\label{def:Phi}
(\varphi,\psi)\in \Phi_p:=\{(\varphi,\psi)\in C_b(\mathbb R);\, \varphi(x)+\psi(X)\leq |x-X|^p\}.
\end{equation}
Finally, for $p=1$, \eqref{def:Wasserstein-dual} can be written as the so-called Kantorovich-Rubinstein formula:
\begin{equation}\label{def:Wasserstein-dual-W1}
W_1(\mu,\nu)=\sup_{\|\psi'\|_\infty\leq 1}\left(\int \psi(x)d\mu(x)-\int\psi(x)\,d\nu(x)\right).
\end{equation}

\medskip

Our first result is a hydrodynamic limit of the model \eqref{eq:model-rep}: we show that when $\gamma>0$ is large, it is possible to describe the dynamics of the %solution is well described by the evolution of the 
macroscopic quantities $N(t)=\int n(t,x)\,dx$ and $Z(t,x)=\int \frac{n(t,x)}{N(t)}\,dx$, which behave like solutions of the following ordinary differential equation:
\begin{equation}\label{eq:EDO}
\left\{\begin{array}{l}
\bar N'(t)=\left(r+\int \alpha(x)\bar u_{\bar Z(t)}(x)\,dx-\beta \bar N(t)\right)\bar N(t),\\
\bar Z'(t)=\left(r+\int \alpha(x)\bar u_{\bar Z(t)}(x)\,dx\right)\left(\int x n_b(x)\,dx-\bar Z(t)\right).
\end{array}\right.
\end{equation}
We also show that the microscopic profile $n(t,x)$ is then characterized by $n(t,x)\sim \bar N(t) \bar u_{\bar Z(t)}$, where $\bar u_Z$ is the unique solution of 
\begin{equation}\label{eq:micro-equi}
\bar u_Z(x)=\int\int K(x,x_1,x_2)\bar u_Z(x_1)\bar u_Z(x_2)\,dx_1\,dx_2,\quad x\in \mathbb R_+,
\end{equation}
such that $\int x\bar u_Z(x)\,dx=Z$.

\begin{theorem}\label{thm:hydro}
Let $M>0$ $K$ as in \eqref{def:K2}, with $B\in \mathcal P_2(\mathbb R_+)$ satisfying \eqref{Ass1}, $\alpha\in W^{1,\infty}( \mathbb R_+)$, $\beta>0$ and $n_b\in \mathcal P_1(\mathbb R_+)$ such that $\int xn_b(x)\,dx<M$. For $Z\in\mathbb R_+$, let $\bar u_Z\in\mathcal P_2(\mathbb R_+)$ be the unique solution of \eqref{eq:micro-equi}, such that $\int x\bar u_Z(x)\,dx=Z$. We define $t\mapsto (\bar N(t),\bar Z(t))\in(\mathbb R_+)^2$ as the solution of \eqref{eq:EDO}, with 
\[(\bar N(0),\bar Z(0))=\left(\int n^0(x)\,dx,\int \frac{xn^0(x)}{\int n^0(x)\,dx}\,dx\right).\]

There exists $C>0$ and $\nu>0$ such that if $n=n(t,x)\in L^\infty_{loc}(\mathbb R_+,L^\infty(\mathbb R_+))$ is a solution of \eqref{eq:model-rep} with initial value $n^0$, then, for $t\geq 0$,
\begin{equation}\label{eq:distance-W2}
 W_2\left(\frac {n(t,\cdot)}{N(t)},u_{\bar Z(t)}\right)+\left|N(t)-\bar N(t)\right|\leq \frac C{\gamma^\nu},
\end{equation}
where $N$ is defined by \eqref{def:N}.
\end{theorem}

\gr{\begin{remark}\label{rem:previous-work}
To study this problem, we will need to first consider a simpler model, where only exchanges are present (see \eqref{eq:model}). The estimates we derive on the exchange operator (Section~4.1) and on the pure exchange model (Section~4.2) have been considered for related models in e.g. \cite{Bisi2009} and \cite{Bassetti2011}, using different methods (either Fourier transform techniques, or probabilist tools). To improve the readability of our study, we have derived all the necessary estimates using Wasserstein distance methods.
\end{remark}}

Our second result is the long-time convergence of solutions of \eqref{eq:model-rep} to a unique steady-state, involving both the exchanges and the birth-death process:

\begin{theorem}\label{thm:hypo}

Let $M>0$, $K$ as in \eqref{def:K2}, with $B\in \mathcal P_2(\mathbb R_+)$ satisfying \eqref{Ass1}, $\alpha\in W^{1,\infty}( \mathbb R_+)$, $\beta,\,\gamma>0$ and $n_b\in \mathcal P_1(\mathbb R_+)$ such that $\int xn_b(x)\,dx<M$. If 
\[\kappa:=r+\min_{x\in \mathbb R_+}\alpha(x)-6M\|\alpha'\|_\infty>0,\]
then there exists a unique measure $x\mapsto \bar N\bar n(x)$, with $\bar n\in\mathcal P_1(\mathbb R_+)$, such that all %for any initial condition $n^0\in \mathcal P_1(\mathbb R_+)$ satisfying $\int xn^0(x)\,dx<M$, the 
solutions $n=n(t,x)\in L^\infty_{loc}(\mathbb R_+,L^\infty(\mathbb R_+))$ of \eqref{eq:model-rep} converge to $\bar N\bar n$ as $t\to\infty$ for the weak-* topology of measures over $\mathbb R_+$, provided $\int xn(0,x)\,dx<M$. More precisely, there exists a constant $C>0$ such that
\[W_1\left(\frac {n(t,\cdot)}{\int n(t,x)\,dx},\bar n\right)\leq Ce^{-\kappa t},\]
\[\left|\int n(t,x)\,dx-\bar N\right|\leq Ce^{-\min\left(\kappa ,\frac{\beta \bar N}4\right)t}.\]
\end{theorem}

\subsection{Related works and discussion}

In \cite{Hinow}, a model was proposed to model the exchange of P-gp in a cell population. The analysis was based on a strong assumption on the nature of exchanges: the result of an exchange between two cells carrying $x_1,\,x_2\in\mathbb R_+$ P-gp originally, was an equilibration of those traits: after the exchange, the cells traits are respectively $x_1'=x_1+f(x_2-x_1)$ and $x_2'=x_2-f(x_2-x_1)$, for some $f\in (0,1)$. After the derivation of this exchange operator, the authors show the existence of solutions for the model (where only exchanges are present), and prove that all solutions converge to a Dirac mass (that is all the individuals ultimately carry the same number of P-gp). The exchange operator derived in Section~\ref{subsec:modelling-exchanges} is an extension of their work, when we assume that no information is exchanged. This assumption has a direct impact on the dynamics of the population: a population where all cells carry the same number of P-gp would be quickly destabilized (in a exchange between two cells containing an equal number of P-gp, it is likely that one cell will send more of its P-gp than the other). Proving the existence of solutions (that we do not consider in the article) should not be more difficulty to prove here than it was in \cite{Hinow} (indeed, the exchange operator we consider is typically more regular). On the contrary,  this new exchange kernel assumption makes the analysis of the dynamics of the solution more delicate: the methods developed in \cite{Hinow} cannot be extended to our problem.

\medskip

In 1978, Tanaka introduced a new entropy functional for the Boltzmann equation. This functional described a contracting effect of the equation's flow, and allowed the author to describe the long-time convergence of solutions. In \cite{Bolley}, this idea was extended to the inelastic Boltzmann equation. We show in this article that the methods developed in the two articles mentioned above can be used to study the dynamics of the solutions of \eqref{eq:model-rep}. The main idea is to consider the Wasserstein distance between two solutions of the equation. It is then possible to show that the exchange operator tends to decrease the distance between solutions: in the case of the $W_2-$Wasserstein distance, the operator decreases the distance between solutions, while for the $W_1-$Wasserstein distance, the operator cannot increase the distance between solutions (eventhough it is not a strict contraction). \gr{The $W_2-$Wasserstein distance then provides a powerful contraction, which we use to derive a hydrodynamic limit of the equation, Theorem~\ref{thm:hydro}. However, the $W_2-$Wasserstein distance does not behave very well when the birth-death process plays a role as important as the exchange process. The $W_1-$Wasserstein distance is then better adapted to the analysis, and it allowed us to prove the long time convergence of solutions of the kinetic model \eqref{eq:model-rep}, see Theorem~\ref{thm:hypo}.}

The Tanaka functional and the wasserstein distance methods employed in our study are indeed related to another set of distances between probability measures, based on the Fourier transform. Those methods were introduced in \cite{Bobylev1,Bobylev2}, and we refer to \cite{Carrillo-rev} for a description of the links existing between those metrics and Wasserstein distances. Those Fourier-based arguments have proven useful to study the properties of the Boltzmann equation in the Maxwellian case, and they could probably also be used to study the dynamics of \eqref{eq:model-rep}. %In 1978, Tanaka introduced a new entropy functional based on the Wasserstein distance for the inelastic Boltzmann equation \cite{Tanaka}, which was then used to study the long time behaviour of solutions in \cite{Bolley}.
The Fourier-based distance was used to study a range of models from econometrics \cite{Bisi2009} and opinion formation theory. Let us finally mention \cite{Bassetti2011}, where probabilist methods are used to prove the contractivity of a large class of kinetic operators.

Another important ingredient of our model is the birth and deaths of individuals, which is impacted by the trait $x$ of individuals. It is a priori difficult to use the tools described above in this context: the solutions are no longer probability measures, and the birth/death are of a different nature than the effect of exchanges (or collision in the case of the Boltzmann equation), which merely moves the mass along the trait $x$. Our strategy is to introduce the size $N(t)$ of the population, and the normalized population $\tilde n(t,x):=\frac{n(t,x)}{N(t)}$. In the special case where $\alpha(x)=ax+b$, it is possible to derive closed equations on $N(t)$ and $Z(t)=\int \frac{n(t,x)}{N(t)}x\,dx$, and the problem is then very close to the situation where only the exchanges are considered (see Theorem~\ref{thm:contraction} and Remark~\ref{rem:simple-alpha}). In more general situations, the dynamics of $N$ $Z$ and $\tilde n$ are coupled. We show however that the contracting effects of the exchange operator (discussed above) and the simple dynamics of the macroscopic quantities $N$ and $Z$ can be combined to study the dynamics of solutions of \eqref{eq:model-rep}.

\bigskip

%
%The exchanges described by \eqref{eq:model-rep} happen through nanotubes between cells (see \cite{Pasquier}). P-gp are indeed also known to be exchanged through the release of microparticles. Those micro-particles, that contain a certain number of P-gp, are then captured by any other cell. If we assume that the receiving cell is chosen uniformly in the population, we believe that this second mechanism could be added to \eqref{eq:model-rep} thanks to a BGK-type operator. The analysis method developed here could then probably be extended to this more complex model. As mentioned in Section~\ref{subsec:modelling-exchanges}, the assumption~\ref{Ass-B} is also a strong biological assumption, that could probably be relaxed.

In this article, we focus our attention on the exchange phenomena: we believe that the exchange term derived in Subsection~\ref{subsec:modelling-exchanges} is biologically convincing model for P-gp transfers happening through nanotubes (see \cite{Pasquier}). Concerning the birth-death process however, many other modelling choices are possible. We decided to consider a simple situation, to avoid additional difficulties to the analysis (although we avoided artificially simple situations, see Remark~\ref{rem:simple-alpha}). Here are the some generalizations that could be interesting in terms of biological applications:
\begin{itemize}
\item If a parent cell carrying $x$ P-glycoproteins gives rise to a daughter cell, we could assume that the traits $x'$ and $y'$ of the parent and daughter cells are given by $x'=y'=\frac{x+x_0}2$, where $x_0$ is drawn from a distribution $x_0$. 
\item A natural extension of this work would be to assume that the P-gp are continuously created by living cells, rather than created at birth only.
\item In this study, we have considered cytotoxic drugs, that are toxic to cells because they increase their death rate. It would also be interesting to consider cytostatic drugs, that affect their birth rates.
\item The exchanges described by \eqref{eq:model-rep} happen through nanotubes between cells (see \cite{Pasquier}). P-gp are indeed also known to be exchanged through the release of microparticles. Those micro-particles, that contain a certain number of P-gp, are then captured by any other cell. This second mechanism could probably be included in the model \eqref{eq:model-rep} thanks to a BGK-type operator.
\end{itemize}
Finally, to study the effect of exchanges on the emergence and propagation of resistance in tumors, one would like to consider several types of cells (with different genotypes), with exchanges happening across the different types. We believe the methods develop here will be helpful to consider those more complicated situations.

\medskip

Finally, let us mention two biological problems linked to the present article. The first class of biological problematic is cooperation in biological populations. Cooperation is a widespread behaviour: it exists at all levels of life, and  appearance and stability of such phenomena is a challenging problem, and we refer to \cite{Marshall,Nowak,Abbot}. One of the mechanisms of cooperation is the exchange of goods. These exchanges could be modelled with an operator similar to the one we derived in Section~\ref{subsec:modelling-exchanges}. The second class of problems is linked to sexual reproduction, where the formation of gametes and recombinations implies that the DNA of the offspring is a combination of the genomes of both parents. A model used in biology to describe the effect of those events on a given phenotype is the so-called infinitesimal model, see \cite{Bulmer,Burger}. The properties of the  operator appearing in this model are similar to the exchange operator described in Section~\ref{subsec:modelling-exchanges}.We believe the analysis methods developed here could also be useful in those two other contexts. In particular, they could be an interesting step towards a rigorous mathematical description of the hydrodynamic limits introduced in \cite{Mirrahimi}.

\section{A Hydrodynamic limit for the model \eqref{eq:model-rep}}
\label{sec:W2}

%In this section, we prove the Theorem~\ref{thm:hydro}. In Subsection~\ref{subsec:dynamics-pure-exchange}, we provide some secondary results when the birth-death process is neglected.

\subsection{Contraction of the exchanges in the $W_2-$distance}
\label{subsec:contractionW2}
The following lemma details the effect of an exchange on the $W_2$-Wasserstein distance between two cells:

\begin{lemma}\label{lem:contraction}
Let $x_1,\,x_2,\,x_1',\,x_2'\in\mathbb R_+^*$, and $K$ as in \eqref{def:K2}. We define $\lambda_1,\,\lambda_2\in [0,1]$ as
\begin{equation}\label{def:lambdai}
\lambda_1:=\int xB(x)\,dx,\quad \lambda_2:=\int x^2B(x)\,dx.
\end{equation}
We have
\begin{align}
&W_2^2\Big(K(\cdot,x_1,x_2),K(\cdot,x_1',x_2')\Big)\leq (1+\lambda_2-2\lambda_1)(x_1-x_1')^2+\lambda_2(x_2-x_2')^2\nonumber\\
&\phantom{W_2^2\Big(K(\cdot,x_1,x_2),K(\cdot,x_1',x_2')\Big)\leq}+2(1-\lambda_1)\lambda_1(x_1-x_1')(x_2-x_2').\label{eq:KW2}
\end{align}
\end{lemma}

\begin{proof}[Proof of Lemma~\ref{lem:contraction}]
To estimate the Wasserstein distance $W_2^2\Big(K(\cdot,x_1,x_2),K(\cdot,x_1',x_2')\Big)$, we will use the Kantorovich dual formula. We consider $(\varphi,\psi)\in \Phi_2$ (see \eqref{def:Phi}), and estimate
\begin{eqnarray*}
I&=&\int \varphi(x)K(x,x_1,x_2)\,dx+\int \psi(X)K(X,x_1',x_2')\,dX\\
&=&\int \varphi(x)\int\int \delta_{x=x_1-y_1+y_2}\frac 1{x_1x_2}B\left(\frac{y_1}{x_1}\right)B\left(\frac{y_2}{x_2}\right)\,dy_1\,dy_2\,dx\\
&&+\int \psi(X)\int\int \delta_{X=x_1'-y_1'+y_2'}\frac 1{x_1'x_2'}B\left(\frac{y_1'}{x_1'}\right)B\left(\frac{y_2'}{x_2'}\right)\,dy_1'\,dy_2'\,dX\\
&=&\int\int  \varphi(x_1-y_1+y_2)\frac 1{x_1x_2}B\left(\frac{y_1}{x_1}\right)B\left(\frac{y_2}{x_2}\right)\,dy_1\,dy_2\\
&&+\int\int  \psi(x_1'-y_1'+y_2')\frac 1{x_1'x_2'}B\left(\frac{y_1'}{x_1'}\right)B\left(\frac{y_2'}{x_2'}\right)\,dy_1'\,dy_2',
\end{eqnarray*}
and then, thanks to the changes of variable $\tilde y_i=\frac{y_i}{x_i}$ and $\tilde y_i'=\frac{y_i'}{x_i}$ for $i=1,\,2$, we get
\begin{eqnarray}
I&=&\int\int \varphi(x_1-x_1y_1+x_2y_2)B\left(y_1\right) B\left(y_2\right)\,dy_1\,dy_2\nonumber\\
&&+\int\int  \psi(x_1'-x_1'y_1'+x_2'y_2')B\left(y_1'\right)B\left(y_2'\right)\,dy_1' \,dy_2'\nonumber\\
&=&\int\int  \left(\varphi(x_1-x_1y_1+x_2y_2)+\psi(x_1'-x_1'y_1+x_2'y_2)\right)B\left(y_1\right) B\left(y_2\right)\,dy_1\,dy_2.\label{eq:I=}
\end{eqnarray}
We can now use the fact that $(\varphi,\psi)\in\Phi_2$ (see \eqref{def:Phi}), and thus
\begin{eqnarray*}
I&\leq&\int\int  \left|(x_1-x_1y_1+x_2y_2)-(x_1'-x_1'y_1+x_2'y_2)\right|^2B\left(y_1\right) B\left(y_2\right)\,dy_1\,dy_2\\
&=&\int\int  \left|(x_1-x_1')(1-y_1)+(x_2-x_2')y_2\right|^2B\left(y_1\right) B\left(y_2\right)\,dy_1\,dy_2\\
&=&\int\int \Big((x_1-x_1')^2(1-y_1)^2+(x_2-x_2')^2y_2^2+2(x_1-x_1')(x_2-x_2')(1-y_1)y_2\Big)\\
&&\qquad B\left(y_1\right) B\left(y_2\right)\,dy_1\,dy_2\\
&=& (x_1-x_1')^2\int(1-y_1)^2B\left(y_1\right) \,dy_1+(x_2-x_2')^2\int y_2^2 B\left(y_2\right)\,dy_2\\
&&+2(x_1-x_1')(x_2-x_2')\left(\int (1-y_1) B\left(y_1\right)\,dy_1\right)\left(\int  y_2B\left(y_2\right)\,dy_2\right),
\end{eqnarray*}
and then, with the notations \eqref{def:lambdai},
\begin{equation*}
I\leq(x_1-x_1')^2(1+\lambda_2-2\lambda_1)+(x_2-x_2')^2\lambda_2+2(x_1-x_1')(x_2-x_2')(1-\lambda_1)\lambda_1.
\end{equation*}
Since this is true for any $(\varphi,\psi)\in\Phi_2$ (see \eqref{def:Phi}), we can use this estimate and \eqref{def:Wasserstein-dual} to show that
\begin{eqnarray*}
W_2^2\Big(K(\cdot,x_1,x_2),K(\cdot,x_1',x_2')\Big)&=&\max_{(\varphi,\psi)\in \Phi_2}I\\
&\leq& (1+\lambda_2-2\lambda_1)(x_1-x_1')^2+\lambda_2(x_2-x_2')^2\nonumber\\
&&+2(1-\lambda_1)\lambda_1(x_1-x_1')(x_2-x_2').
\end{eqnarray*}

\end{proof}

We now show that the exchange dynamics has a contraction effect on the population, in the sense of the $W_2-$distance:

\begin{lemma}\label{lem:contractionW2}
Let $K$ defined by \eqref{def:K2}, with $B\in \mathcal P_2(\mathbb R_+)$ satisfying \eqref{Ass1}. Then, for any $u_1,\,u_2\in \mathcal P_2(\mathbb R_+)$ such that
\begin{equation}\label{eq:1ermoment}
\int x u_1(x)\,dx=\int x u_2(x)\,dx,
\end{equation}
we have
\begin{equation}\label{eq:estK1}
W_2(u_1',u_2')\leq \left(1+2\int x(x-1)B(x)\,dx\right)^{\frac 12}W_2\left(u_1,u_2\right),
\end{equation}
where $u_i'(x):=\int K(x,x_1,x_2) u_i(x_1)u_i(x_2)\,dx_1\,dx_2$ for $i=1,\,2$. Moreover, 
\[\left(1+2\int x(x-1)B(x)\,dx\right)^{\frac 12}<1.\]
\end{lemma}

\begin{proof}[Proof of Lemma~\ref{lem:contractionW2}]
We want to estimate the $W_2$ distance between $u_1'$ and $u_2'$ thanks to \eqref{def:Wasserstein-dual}. We consider $(\varphi,\psi)\in\Phi_2$, and estimate
\begin{eqnarray*}
I&:=&\int \varphi(x)u_1'(x)\,dx+\int \psi(X)u_2'(X)\,dX\nonumber\\
&=&\int \varphi(x)\int\int K(x,x_1,x_2)u_1(x_1)u_1(x_2)\,dx_1\,dx_2\,dx\\
&&+\int \psi(X)\int\int K(X,x_1',x_2')u_2(x_1')u_2(x_2')\,dx_1'\,dx_2'\,dX\\
&=&\int\int \left(\int \varphi(x)K(x,x_1,x_2)\,dx\right) u_1(x_1)u_1(x_2)\,dx_1\,dx_2\\
&&+\int\int \left(\int \psi(X)K(X,x_1',x_2')\,dX\right)u_2(x_1')u_2(x_2')\,dx_1'\,dx_2'\\
\end{eqnarray*}
Remember that $u_1,u_2 \in \mathcal P_2(\mathbb R_+)$ implies $\int u_1(x)dx=\int u_2(x)dx=1$. It follows that 
\begin{equation*}
\begin{array}{l}
I=  \int_{\mathbb R_+^4} \left(\int \varphi(x)K(x,x_1,x_2)\,dx+\int \psi(X)K(X,x_1',x_2')\,dX\right)u_1(x_1)u_1(x_2)u_2(x_1')u_2(x_2')\,dx_1\,dx_2\,dx_1'\,dx_2'\\
\end{array}
\end{equation*}
Let now $\pi_t\in \Pi(u_1,u_2)$ (see \eqref{def:Wasserstein}). The above equality can be rewritten as follow
\begin{align}
&I = \int_{\mathbb R_+^4} \left(\int \varphi(x)K(x,x_1,x_2)\,dx+\int \psi(X)K(X,x_1',x_2')\,dX\right) \,d\pi_t(x_1,x_1')\,d\pi_t(x_2,x_2')\nonumber\\
&\quad\leq \int_{\mathbb R_+^4}  W_2^2\left(K(\cdot,x_1,x_2),K(\cdot,x_1',x_2')\right) \,d\pi_t(x_1,x_1')\,d\pi_t(x_2,x_2'),\label{est:L}
\end{align}
where we have used the fact that $(\varphi,\psi)\in\Phi_2$ as in formula \eqref{def:Wasserstein-dual}. We now use the result of Lemma~\ref{lem:contraction} for $W_2$ to obtain
\begin{align}
&I\leq\int\int \left((1+\lambda_2-2\lambda_1)(x_1-x_1')^2+\lambda_2(x_2-x_2')^2+2(1-\lambda_1)\lambda_1(x_1-x_1')(x_2-x_2')\right)\nonumber\\
&\qquad \,d\pi_t(x_1,x_1')\,d\pi_t(x_2,x_2')\nonumber\\
&\quad\leq(1+\lambda_2-2\lambda_1)\int\left(\int (x_1-x_1')^2\,d\pi_t(x_1,x_1')\right)\,d\pi_t(x_2,x_2')\nonumber\\
&\qquad+\lambda_2\int\left(\int (x_2-x_2')^2\,d\pi_t(x_2,x_2')\right)\,d\pi_t(x_1,x_1')\nonumber\\
&\qquad+2(1-\lambda_1)\lambda_1\left(\int(x_1-x_1')\,d\pi_t(x_1,x_1')\right)\left(\int(x_2-x_2')\,d\pi_t(x_2,x_2')\right).\label{eq:differents-bar-u1}
\end{align}
We notice next that thanks to \eqref{eq:1ermoment},
\begin{equation}\label{eq:differents-bar-u2}
\int(x-x')\,d\pi_t(x,x')=\int xu_1(x)\,dx-\int xu_2(x)\,dx=0,
\end{equation}
and then, since this estimate holds for any  $(\varphi,\psi)\in\Phi_2$, thanks to the Kantorovich dual formula \eqref{def:Wasserstein-dual},
\begin{eqnarray*}
W_2^2(u_1',u_2')&\leq&(1+2\lambda_2-2\lambda_1)\int (x_1-x_1')^2\,d\pi_t(x_1,x_1').
\end{eqnarray*}
Since this inequality holds for any $\pi_t\in \Pi(u_1,u_2)$ (with the notation of \eqref{def:Wasserstein}), we can take the minimum over such $\pi_t$, and get
\begin{eqnarray*}
W_2^2(u_1',u_2')&\leq& (1+2\lambda_2-2\lambda_1)W_2^2\left(u_1,u_2\right)\\
&=&\left(1+2\int x(x-1)B(x)\,dx\right)W_2^2\left(u_1,u_2\right).
\end{eqnarray*}
Finally, $\left(1+2\int x(x-1)B(x)\,dx\right)^{\frac 12}<1$ is a direct consequence of \eqref{Ass1}.
\end{proof}
%
%The contraction property of the operator comes from the following property (that is trivial):
%
%\begin{lemma}\label{lem:lambdai}
%Let $K$ defined by \eqref{def:K2}, and $\lambda_1$ defined by \eqref{def:lambdai}. Then,then 
%\[1+2\lambda_2-2\lambda_1=1+2\int x(x-1)B(x)\,dx,\]
%and in particular, $1+2\lambda_2-2\lambda_1<1$ as soon as $B$ satisfies \eqref{Ass1}.
%\end{lemma}

\subsection{Steady-states of the pure exchange model}
\label{subsec:steady-states-pure-exchange}

We consider a model where the population is only affected by exchanges, that is

\begin{equation}\label{eq:model}
\partial_t u(t,x)=\int\int K(x,x_1,x_2)u(t,x_1)u(t,x_2)\,dx_1\,dx_2-u(t,x),
\end{equation}
where $K$ is defined by \eqref{def:K2}. The first result on this model is the existence of steady-states:
\begin{theorem}\label{thm:steady-states}
Let $K$ defined by \eqref{def:K2}, with $B\in \mathcal P_2(\mathbb R_+)$ satisfying \eqref{Ass1}, and $Z\in\mathbb R_+$. There is a unique steady-state $\bar u$ of \eqref{eq:model} such that $\int x \bar u(x)\,dx=Z$.

Moreover, there exists $C>0$ such that for any $Z_1, \,Z_2\geq 0$,
\begin{equation}\label{eq:differents-bar-u0}
W_2(\bar u_{Z_1},\bar u_{Z_2})\leq C|Z_1-Z_2|.
\end{equation}
\end{theorem}
\begin{proof}[Proof of Theorem~\ref{thm:steady-states}]
Notice that $\left\{u\in\mathcal P_2(\mathbb R_+);\,\int x  u(x)\,dx=Z\right\}$ is a closed subset of $\mathcal P_2(\mathbb R_+)$ for $W_2$, and then $\left(\left\{u\in\mathcal P_2(\mathbb R_+);\,\int x  u(x)\,dx=Z\right\},W_2\right)$ is a complete metric space (see e.g. \cite{Villani}). Let
\begin{equation}\label{eq:T}
T: u\in\mathcal P_2(\mathbb R_+) \mapsto \left(\int\int K(\cdot,x_1,x_2) u(x_1)u(x_2)\,dx_1\,dx_2\right).
\end{equation}
\gr{More precisely, for any $h\in L^\infty(\mathbb R_+)$, thanks to \eqref{def:K3},
\begin{eqnarray*}
\int h(x)T(u)(x)\,dx&:=&\int\left(\int \left(\int\left(\int h(x_1-y_1+y_2) B\left(\frac{y_1}{x_1}\right)\,dy_1\right)B\left(\frac{y_2}{x_2}\right)\,d y_2\right) u(x_1)\,dx_1\right)u(x_2)\,dx_2\\
&\leq& \|h\|_{L^\infty},
\end{eqnarray*}
and}
\begin{eqnarray*}
\int x^2T(u)(x)\,dx&=& \int \int \left(\int x^2K(x,x_1,x_2)\,dx\right) u(x_1)u(x_2)\,dx_1\,dx_2\\
&=& \int \int \left(\int\int \left(x_1-y_1+y_2\right)^2B\left(\frac{y_1}{x_1}\right)B\left(\frac{y_2}{x_2}\right)\,dy_1\,dy_2\right) \\
&&\qquad u(x_1)u(x_2)\,dx_1\,dx_2\\
&=&\int \int \left(x_1^2(1+\lambda_2-2\lambda_1)+x_2^2\lambda_2+ 2x_1x_2\left(\lambda_1-\lambda_1^2\right)\right) \\
&&\qquad u(x_1)u(x_2)\,dx_1\,dx_2\\
&=&\left(\int x^2u(x)\,dx\right)(1+2\lambda_2-2\lambda_1)+ 2Z^2\left(\lambda_1-\lambda_1^2\right)<\infty.
\end{eqnarray*} 
\gr{Note that for any $x_1,\,x_1\geq 0$, $\textrm{supp }K(\cdot,x_1,x_2)$ is bounded, the unbounded test function $x\mapsto x^2$ can thus be modified to become an $L^\infty$ test function, for which $\int x^2K(x,x_1,x_2)\,dx$ is well defined. The computation above is thus valid.
} We can thus define $T$ as an operator mapping $\mathcal P_2(\mathbb R_+)$ into itself. Moreover, for any $u\in \left\{u\in\mathcal P_2(\mathbb R_+);\,\int x  u(x)\,dx=Z\right\}$,
\begin{eqnarray*}
\int xT(u)(x)\,dx&=& \int \int \left(\int xK(x,x_1,x_2)\,dx\right) u(x_1)u(x_2)\,dx_1\,dx_2\\
&=& \int \int \left(\int\int \left(x_1-y_1+y_2\right)B\left(\frac{y_1}{x_1}\right)B\left(\frac{y_2}{x_2}\right)\,dy_1\,dy_2\right) \\
&&\qquad u(x_1)u(x_2)\,dx_1\,dx_2\\
&=&\int \int \left(x_1-x_1\lambda_1+x_2\lambda_1\right) \\
&&\qquad u(x_1)u(x_2)\,dx_1\,dx_2\\
&=&\int xu(x)\,dx=Z,
\end{eqnarray*} 
so that the application $T$ maps $\left\{u\in \mathcal P_2(\mathbb R_+);\,\int x \bar u(x)\,dx=Z\right\}$ into itself. Finally, thanks to Lemma~\ref{lem:contractionW2}, the application is a strict contraction on this set. We can then apply the Banach fixed point Theorem, to show that there exists then a unique measure $\bar u\in\left\{u\in\mathcal P_2(\mathbb R_+);\,\int x  u(x)\,dx=Z\right\}$, such that 
\[\forall x\in\mathbb R_+,\quad \bar u=T(\bar u).\]
$\bar u$ is then the unique steady-state of \eqref{eq:model} such that $\int x \bar u(x)\,dx=Z$.

\medskip

 We can reproduce the proof of Lemma~\ref{lem:contractionW2} until \eqref{eq:differents-bar-u1} with $u_1:=\bar u_{Z_1}$ and $u_2:=\bar u_{Z_2}$. Here, \eqref{eq:differents-bar-u2} becomes
\begin{equation*}%\label{eq:differents-bar-u2}
\int(x-x')\,d\pi_t(x,x')=\int xu_1(x)\,dx-\int xu_2(x)\,dx=Z_1-Z_2,
\end{equation*}
and then, since this estimate holds for any  $(\varphi,\psi)\in\Phi_2$, thanks to the Kantorovich dual formula \eqref{def:Wasserstein-dual},
\begin{eqnarray*}
W_2^2(u_1',u_2')&\leq&(1+2\lambda_2-2\lambda_1)\int (x_1-x_1')^2\,d\pi_t(x_1,x_1')+2(1-\lambda_1)\lambda_1|Z_1-Z_2|^2.
\end{eqnarray*}
Since this inequality holds for any $\pi_t\in \Pi(u_1,u_2)$ (with the notation of \eqref{def:Wasserstein}), we can take the minimum over such $\pi_t$, and get
\begin{equation*}
W_2^2(u_1',u_2')\leq (1+2\lambda_2-2\lambda_1)W_2^2\left(u_1,u_2\right)+2(1-\lambda_1)\lambda_1|Z_1-Z_2|^2.
\end{equation*}
Now, $u_1=\bar u_{Z_1}$ and $u_2=\bar u_{Z_2}$ are steady-points of $T$, and then, $u_1'=u_1$, $u_2'=u_2$. Thus,
\begin{equation*}
W_2^2(\bar u_{Z_1},\bar u_{Z_2})\leq \frac {(1-\lambda_1)\lambda_1}{\lambda_1-\lambda_2}|Z_1-Z_2|^2.
\end{equation*}

\end{proof}

\subsection{Dynamics of the pure exchange model}
\label{subsec:dynamics-pure-exchange}

We show here that the equation \eqref{eq:model} induces a contraction for the $W_2-$distance, which implies in particular that all solutions converge to the steady-state given by Theorem~\ref{thm:steady-states}.

% \textcolor{red}{ Dans theoreme ci-dessous j'ai suprimer une partie inutile (laisser en commentaire dans le fichier latex). Par contre il manque une partie existence et unicite des solutions, ainsi que la preservation du nombre d'individus et de la quantite total de P-gp. }

\begin{theorem}\label{thm:contraction}
Let $K$ defined by \eqref{def:K2}, with $B\in \mathcal P_2(\mathbb R_+)$ satisfying \eqref{Ass1}, and $Z\in\mathbb R_+$. %Then, any 

Let two non negative solutions $u=u(t,x)\in L^\infty_{loc}(\mathbb R_+,L^\infty(\mathbb R_+))$ and $v=v(t,x)\in L^\infty_{loc}(\mathbb R_+,L^\infty(\mathbb R_+))$ of \eqref{eq:model} such that 
\[Z=\int x u(0,x)\,dx=\int x v(0,x)\,dx.\]
Then, $u$ and $v$ satisfy 
\begin{equation}\label{eq:uv-W2}
W_2^2\left(u(t,\cdot),v(t,\cdot)\right)\leq W_2^2\left(u(0,\cdot),v(0,\cdot)\right)e^{-c t},
\end{equation}
with $c=\int x(1-x)B(x)\,dx>0$.

In particular, any solution $u=u(t,x)$ of \eqref{eq:model} such that $\int x u(0,x)\,dx=Z$ converges to the steady-state $\bar u_Z$ defined in Theorem~\ref{thm:steady-states}:
\begin{equation}\label{eq:u-W2}
W_2^2\left(u(t,\cdot),\bar u_Z\right)\leq W_2^2\left(u(0,\cdot),\bar u_Z\right)e^{-c t},
\end{equation}
where $c=\int x(1-x)B(x)\,dx>0$.
\end{theorem}

\begin{remark}\label{rem:simple-alpha}
If $\alpha(x)=ax+b$, the dynamics of $N(t)$ and $Z(t)=\int x\frac{n(t,x)}{N(t)}\,dx$ only depends on $N$ and $Z$, and not on the detailed distribution $n$ of the population. It is then possible to show that $N$ and $Z$ converge exponentially to a limit $(\bar N,\bar Z)$. We can then use the contracting  properties of the exchange operator, just as in the proof of Theorem~\ref{thm:contraction} to prove the convergence of the solution $n$ to a unique limit. If $\alpha(\cdot)$ is more complicated however, it is not possible to write close equations on $N$ and $Z$, and to describe the long-time dynamics of the solution then, we will need to use a different metric, namely the $W_1-$Wasserstein distance, see Theorem~\ref{thm:hypo}.
\end{remark}

\begin{proof}[Proof of Theorem~\ref{thm:contraction}]

We now consider solutions $t\mapsto u(t,\cdot)\in\mathcal P_2(\mathbb R_+)$ of \eqref{eq:model}, which write
\[u(t,x)=e^{-t}u(0,x)+\int_0^te^{-(t-s)}\int\int K(x,x_1,x_2)u(t,x_1)u(t,x_2)\,dx_1\,dx_2.\]
Let $u$ and $v$ be two solutions of \eqref{eq:model} with respective initial conditions $u_0$, $v_0$. We want to estimate the $W_2$ distance between $u$ and $v$ thanks to \eqref{def:Wasserstein-dual}. We then consider $(\varphi,\psi)\in\Phi_2$, and estimate
\begin{eqnarray}
I&=&\int \varphi(x)u(t,x)\,dx+\int \psi(X)v(t,X)\,dX\nonumber\\
&=&e^{-t}\left(\int \varphi(x)u(0,x)\,dx+\int \psi(X)v(0,X)\,dX\right)\nonumber\\
&&+\int_0^t e^{-(t-s)}\Bigg[\int \varphi(x)\int\int K(x,x_1,x_2)u(t,x_1)u(t,x_2)\,dx_1\,dx_2\,dx\nonumber\\
&&\qquad+\int \psi(X)\int\int K(X,x_1',x_2')v(t,x_1')v(t,x_2')\,dx_1'\,dx_2'\,dX\Bigg]\,ds\nonumber\\
&\leq& e^{-t}W_2^2\left(u(0,\cdot),v(0,\cdot)\right)+\int_0^t (1+2\lambda_2-2\lambda_1)W_2^2(u(s,\cdot),v(s,\cdot))\,ds,\nonumber
\end{eqnarray}
thanks to the Kantorovich formula \eqref{def:Wasserstein-dual} and Lemma~\ref{lem:contractionW2}. Since this holds for any $(\varphi,\psi)\in \Phi_2$ (see \eqref{def:Wasserstein-dual}), we can pass to the suppremum over such functions to show that
\[W_2^2\left(u(t,\cdot),v(t,\cdot)\right)\leq e^{-t}W_2^2\left(u(0,\cdot),v(0,\cdot)\right)+(1+2\lambda_2-2\lambda_1)\int_0^t e^{s-t}W_2^2\left(u(s,\cdot),v(s,\cdot)\right)\,ds.\]
This estimate can be written
\[y(t)\leq W_2^2\left(u(0,\cdot),v(0,\cdot)\right)+\kappa\int_0^t y(s)\,ds,\]
where $y(t):=e^tW_2^2\left(u(t,\cdot),v(t,\cdot)\right)$. Then, thanks to the Gronwall inequality, $y(t)\leq W_2^2\left(u(0,\cdot),v(0,\cdot)\right)e^{(1+2\lambda_2-2\lambda_1) t}$, and \eqref{eq:uv-W2} follows. Finally, when $v=\bar u_Z$ is a steady-state of \eqref{eq:model}, \eqref{eq:uv-W2} proves the exponential convergence of any solution (such that $\int xu(0,x)\,dx=Z$) of \eqref{eq:model} to $\bar u_Z$. The proof of Theorem~\ref{thm:contraction} is then complete.

\end{proof}

\subsection{Proof of Theorem~\ref{thm:hydro}}
\label{subsec:proof-thm-asymptotic}

\noindent\emph{Step 1: Equations on $N(t)$, $Z(t)$, and $\frac{n(t,\cdot)}{N(t)}$}

For $t\geq 0$, we define $N(t)$ by \eqref{def:N}, and $Z(t)$ as
\begin{equation}\label{def:Z}
Z(t):=\int x\,\frac{n(t,x)}{N(t)}\,dx.
\end{equation}
We can integrate the equation \eqref{eq:model-rep} to get the following equation on $N$:
\begin{equation}\label{eq:N2}
\frac d{dt}N(t)=\left(r+\int \alpha(x)\frac{n(t,x)}{N(t)}\,dx-\beta N(t)\right)N(t),
\end{equation}
which implies that for all $t\geq 0$, $0\leq N(t)\leq \max\left(N(0),\frac{r+\|\alpha\|_\infty}\beta\right)$, and
\[\left\|N'(t)\right\|_\infty\leq C,\]
where the constant $C>0$ is independent of $\gamma>0$. If we derive $Z$, we get
\begin{eqnarray}
\frac d{dt}Z(t)&=&\frac 1{N(t)}\int x\partial_t n(t,x)\,dx-\frac {N'(t)}{N(t)^2} \int x n(t,x)\,dx\nonumber\\
&=&\left(r+\int \alpha(x)\frac{n(t,x)}{N(t)}\,dx\right)Z_b-\beta N(t) Z(t)\nonumber\\
&&-\left(r+\int \alpha(x)\frac{n(t,x)}{N(t)}\,dx-\beta N(t)\right)Z(t)\nonumber\\
&=&\left(r+\int \alpha(x)\frac{n(t,x)}{N(t)}\,dx\right)(Z_b-Z(t)),\label{eq:Z2}
\end{eqnarray}
where $Z_b:=\int x n_b(x)\,dx$. Since $\alpha\geq 0$, we have that for all $t\geq 0$, $0\leq Z(t)\leq \max\left(\int x\frac{n(0,x)}{N(0)}\,dx,Z_b\right)$, and 
\[\left\|Z'(t)\right\|_\infty\leq C,\]
where the constant $C>0$ is independent of $\gamma>0$.

We introduce the following normalized population:
\begin{equation}\label{def:tilden}
\tilde n(t,x):= \frac {n(t,x)}{N(t)},
\end{equation}
which satisfies
\begin{eqnarray}
\partial_t \tilde n(t,x)&=&\frac 1{N(t)}\partial_t n(t,x)-\frac {n(t,x)}{\left(\int n(t,y)\,dy\right)^2}\frac d{dt}\int n(t,x)\,dx\nonumber\\
&=&\left(r+\int \alpha (y) \tilde n(t,y)\,dy\right)n_b(x)-\beta n(t,x)\nonumber\\
&&+\frac \gamma{N(t)^2}\int\int K(x,x_1,x_2)n(t,x_1) n(t,x_2)\,dx_1\,dx_2\nonumber\\
&&-\frac \gamma{N(t)}n(t,x)\nonumber\\
&&-\left[\left(r+\int \alpha (y) \tilde n(t,y)\,dy\right)-\beta N(t)\right]\frac {n(t,x)} {N(t)},\nonumber
\end{eqnarray}
that is
\begin{eqnarray}
\partial_t \tilde n(t,x)&=&\left(r+\int \alpha (y) \tilde n(t,y)\,dy\right)\left(n_b(x)-\tilde n(t,x)\right)\nonumber\\
&&+\gamma \left[\int\int K(x,x_1,x_2)\tilde n(t,x_1) \tilde n(t,x_2)\,dx_1\,dx_2-\tilde n(t,x)\right].\label{eq:model-tilde-n}
\end{eqnarray}

\medskip

\noindent\emph{Step 2: uniform bound on the second moment of $n(t,\cdot)$}

%
%Let $u\in \mathcal P_2(\mathbb R_+)$ and $Z:=\int x u(x)\,dx$. We use Lemma~\ref{lem:contractionW2} to show
%\begin{eqnarray*}
%\sqrt{\int x^2 T(u)(x)\,dx}&=&W_2(T(u),\delta_0)\leq W_2(Tu,T(u_Z))+W_2(u_Z,\delta_0)\\
%&\leq& \left(1+2\int x(x-1)B(x)\,dx\right)W_2(u,u_Z)+C\\
%%&\leq& \left(1+2\int x(x-1)B(x)\,dx\right)W_2(u,\delta 0)+C\\
%&\leq& \left(1+2\int x(x-1)B(x)\,dx\right)\sqrt{\int x^2 u(x)\,dx}+C,
%\end{eqnarray*}
%and then $\int x^2 T(u)(x)\,dx\leq \kappa \int x^2 u(x)\,dx+C$, for some $\kappa<0$.

$\tilde n$ satisfies
\begin{eqnarray*}
\tilde n(t,x)&=&\tilde n(0,x)e^{-\gamma t}\\
&&+\int_0^{t} e^{-\gamma(t-s)}\bigg[\left(r+\int \alpha (y) \tilde n(s,y)\,dy\right)(n_b(x)-\tilde n(s,x))\\
&&\qquad +\gamma T(\tilde n(s,\cdot))\bigg]\,ds,
\end{eqnarray*}
and then
\begin{eqnarray}
\int \tilde n(t,x)x^2\,dx&=&\left(\int \tilde n(0,x)x^2\,dx\right) e^{-\gamma t}\nonumber\\
&&+\int_0^{t} e^{-\gamma(t-s)}\bigg[\left(r+\int \alpha (y) \tilde n(s,y)\,dy\right)\left(\int \tilde n_b(x)x^2\,dx-\int \tilde n(s,x)x^2\,dx\right)\nonumber\\
&&\qquad +\gamma \int T(\tilde n(s,\cdot))(x)x^2\,dx\bigg]\,ds.\label{est:m21}
\end{eqnarray}
We use Lemma~\ref{lem:contractionW2} to estimate the last term: since $Z(t)$ is uniformly bounded,
\begin{eqnarray*}
\sqrt{\int x^2 T(\tilde n(s,\cdot))(x)\,dx}&=&W_2(T(\tilde n(s,\cdot)),\delta_0)\leq W_2(T(\tilde n(s,\cdot)),T(u_{Z(s)}))+W_2(u_{Z(s)},\delta_0)\\
&\leq& \left(1+2\int x(x-1)B(x)\,dx\right)W_2(\tilde n(s,\cdot),u_{Z(s)})+C\\
%&\leq& \left(1+2\int x(x-1)B(x)\,dx\right)W_2(u,\delta 0)+C\\
&\leq& \left(1+2\int x(x-1)B(x)\,dx\right)\sqrt{\int x^2 \tilde n(s,x)\,dx}+C,
\end{eqnarray*}
and then $\int x^2 T(\tilde n(s,\cdot))(x)\,dx\leq \kappa \int x^2 \tilde n(s,x)\,dx+C$, for some $\kappa\in(0,1)$. If we inject this estimate in \eqref{est:m21}, we get
\begin{eqnarray*}
\int \tilde n(t,x)x^2\,dx &\leq& C+\int_0^{t} e^{-\gamma(t-s)}\kappa\gamma \left(\int\tilde n(s,x)x^2\,dx\right)\,ds\\
&\leq& C+\kappa \left(\sup_{s\in[0,t]}\int\tilde n(s,x)x^2\,dx\right)\gamma\int_0^t e^{-\gamma(t-s)}\,ds\\
&\leq& C+\kappa \left(\sup_{s\in[0,t]}\int\tilde n(s,x)x^2\,dx\right).
\end{eqnarray*}
Thus, for all $t\geq 0$,
\begin{equation}\label{eq:2bound}
\int \tilde n(t,x)x^2\,dx\leq \max\left(\int \tilde n(0,x)x^2\,dx,\frac C{1-\kappa}\right).
\end{equation}
Note that this bound implies also a bound on $W_2^2(\tilde n(t,\cdot),\delta_0)=\int \tilde n(t,x)x^2\,dx$. Finally, we notice that this argument can be reproduced for the equation \eqref{eq:model}, with the initial condition $u(0,\cdot)=\delta_Z$. Then \eqref{eq:2bound} becomes $\int u(t,x)x^2\,dx\leq \max\left(Z^2,\frac C{1-\kappa}\right)$, and since $u$ converges to $\bar u_Z$ thanks to Theorem~\ref{thm:contraction}, we get
\begin{equation}\label{eq:2bound-u}
\int u_Z(x)x^2\,dx\leq C(Z^2+1),
\end{equation}
which, here also, is equivalent to $W_2(\tilde n(t,\cdot),\delta_0)\leq C(Z+1)$
%
%\newpage
%
%A VOIR : borne moment d'ordre 2
%\begin{equation}\label{eq:2bound}
%\int x^2 n_b(x)\,dx\leq C.
%\end{equation}
%
%We can deduce from this rough estimate the following bound:
%\begin{equation}\label{eq:boundW2}
%W_2()<C
%\end{equation}

%
%\medskip
%
%\noindent\emph{Step 3: Equation on a normalized population}
%
%We compute first the time derivative of $N$:
%\begin{eqnarray}
%\partial_t N(t)&=&\int \partial_t n(t,x)\,dx\nonumber\\
%&=& \left(r +\int \alpha (y) \frac{n(t,y)}{N(t)}\,dy-\beta N(t)\right)N(t)\label{eq:N}
%\end{eqnarray}

%
%
%We introduce the following normalized population:
%\begin{equation}\label{def:tilden}
%\tilde n(t,x):= \frac {n(t,x)}{N(t)},
%\end{equation}
%which satisfies
%\begin{eqnarray}
%\partial_t \tilde n(t,x)&=&\frac 1{N(t)}\partial_t n(t,x)-\frac {n(t,x)}{\left(\int n(t,y)\,dy\right)^2}\frac d{dt}\int n(t,x)\,dx\nonumber\\
%&=&\left(r+\int \alpha (y) \tilde n(t,y)\,dy\right)n_b(x)-\beta n(t,x)\nonumber\\
%&&+\frac \gamma{N(t)^2}\int\int K(x,x_1,x_2)n(t,x_1) n(t,x_2)\,dx_1\,dx_2\nonumber\\
%&&-\frac \gamma{N(t)}n(t,x)\nonumber\\
%&&-\left[\left(r+\int \alpha (y) \tilde n(t,y)\,dy\right)-\beta N(t)\right]\frac {n(t,x)} {N(t)},\nonumber
%\end{eqnarray}
%that is
%\begin{eqnarray}
%\partial_t \tilde n(t,x)&=&\left(r+\int \alpha (y) \tilde n(t,y)\,dy\right)\left(n_b(x)-\tilde n(t,x)\right)\nonumber\\
%&&+\gamma \left[\int\int K(x,x_1,x_2)\tilde n(t,x_1) \tilde n(t,x_2)\,dx_1\,dx_2-\tilde n(t,x)\right].\label{eq:model-tilde-n}
%\end{eqnarray}
%
%

\medskip

\noindent\emph{Step 3: Estimates on $W_2\left(\tilde n(t,\cdot),\bar u_{Z(t)}\right)$}
%
%Let:
%\begin{equation}\label{def:omega1}
%\omega(t):=r+\frac 12\left(\int \alpha (y) \tilde n(t,y)\,dy+\int \alpha (y) \bar u_{Z(t)}(t,y)\,dy\right).
%\end{equation}
For $0\leq t<t+\tau$, $\tilde n$ can be written
\begin{eqnarray*}
\tilde n(t+\tau,x)&=&\tilde n(t,x)e^{-\gamma \tau}\\
&&+\int_t^{t+\tau} e^{-\gamma(t+\tau-s)}\bigg[\left(r+\int \alpha (y) \tilde n(s,y)\,dy\right)(n_b(x)-\tilde n(s,x))\\
&&\qquad +\gamma T(\tilde n(s,\cdot))\bigg]\,ds,
\end{eqnarray*}
where we have used the notation $T$ introduced in \eqref{eq:T}. Since $\bar u_{Z(t)}$ is a steady-state of \eqref{eq:model}, it satisfies
\begin{eqnarray*}
\bar u_{Z(t)}(x)&=&\bar u_{Z(t)}(x)e^{-\gamma\tau}\\
&&+\int_t^{t+\tau} e^{-\gamma(t+\tau-s)}\gamma T(\bar u_{Z(t)})(x)\,ds.
\end{eqnarray*}

We consider $(\varphi,\psi)\in\Phi_2$, and estimate
\begin{eqnarray*}
I&=&\int \varphi(x)\tilde n(t+\tau,x)\,dx+\int \psi(X)\bar u_{Z(t)}(X)\,dX\nonumber\\
&=&e^{-\gamma\tau}\left(\int \varphi(x)\tilde n(t,x)\,dx+\int \psi(X)\bar u_{Z(t)}(X)\,dX\right)\\
&&+\int_t^{t+\tau} e^{-\gamma(t+\tau-s)}\int \varphi(x)\left(r+\int \alpha (y) \tilde n(s,y)\,dy\right)(n_b(x)-\tilde n(s,x))\,dx\,ds\\
&&+\int_t^{t+\tau} e^{-\gamma(t+\tau-s)}\gamma\bigg(\int \psi(X)T(\tilde n(s,\cdot))+\int \varphi(x)T(\bar u_{Z(T)})(X)\,dX\bigg)\,ds.
\end{eqnarray*}
We can now use the formula \eqref{def:Wasserstein-dual} to get:
\begin{eqnarray}
I&\leq&e^{-\gamma\tau}W_2^2\left(\tilde n(t,\cdot),\bar u_{Z(t)}\right)\nonumber\\
&&+\int_t^{t+\tau} e^{-\gamma(t+\tau-s)}\left(r+\int \alpha (y) \tilde n(s,y)\,dy\right)\int \varphi(x)(n_b(x)-\tilde n(s,x))\,dx\,ds\nonumber\\
&&+\int_t^{t+\tau} e^{-\gamma(t+\tau-s)}\gamma W_2^2\left(T(\tilde n(s,\cdot)),T(\bar u_{Z(t)})\right)\,ds.\label{eq:estItruc}
\end{eqnarray}
To estimate the second term of this expression, we notice that for any couple $(\varphi,\psi)\in\Phi_2$, we have $-\varphi(y)\leq \psi(y)$, and then,
\begin{eqnarray*}
\int \varphi(x)(n_b(x)-\tilde n(s,x))\,dx&\leq& \int \varphi(x)n_b(x)\,dx+\int\psi(X)\tilde n(s,X)\,dX\\
&\leq& W_2^2(n_b,\tilde n(s,\cdot))\\
&\leq& \int x^2 n_b(x)\,dx+\int x^2 n(s,x)\,dx\leq C,
\end{eqnarray*}
thanks to the uniform bound \eqref{eq:2bound} obtained in Step~2. We can use this estimate as well as Lemma~\ref{lem:contractionW2} to estimate \eqref{eq:estItruc}, and if we additionally pass to the suppremum over $(\varphi,\psi)\in\Phi_2$, we obtain
\begin{eqnarray*}
W_2^2\left(\tilde n(t+\tau,\cdot),\bar u_{Z(t)}\right)&\leq&e^{-\gamma\tau}W_2^2\left(\tilde n(t,\cdot),\bar u_{Z(t)}\right)\\
&&+C(r+\|\alpha\|_\infty)\int_t^{t+\tau} e^{-\gamma(t+\tau-s)}\,ds\\
&&+(1+2\lambda_2-2\lambda_1)\gamma\int_t^{t+\tau} e^{-\gamma(t+\tau-s)} W_2^2\left(\tilde n(s,\cdot),\bar u_{Z(t)}\right)\,ds.
\end{eqnarray*}
We can then introduce, for $\sigma\geq 0$, $y(t+\sigma):=e^{\gamma\sigma}W_2^2\left(\tilde n(t+\sigma,\cdot),\bar u_{Z(t)}\right)$, that satisfies
\begin{eqnarray*}
y(t+\tau)&\leq&y(t)+\frac C\gamma\\
&&+(1+2\lambda_2-2\lambda_1)\gamma\int_0^{\tau} y(t+s)\,ds,
\end{eqnarray*}
and then, thanks to the Gronwall inequality,
\[y(t+\tau)\leq \left(y(t)+\frac C\gamma\right)e^{(1+2\lambda_2-2\lambda_1)\gamma\tau},\]
that is
\[W_2^2\left(\tilde n(t+\tau,\cdot),\bar u_{Z(t)}\right)\leq \left(W_2^2\left(\tilde n(t,\cdot),\bar u_{Z(t)}\right)+\frac {C}\gamma\right)e^{-2(\lambda_1-\lambda_2)\gamma\tau},\]
or
\begin{equation}\label{eq:chose2}
W_2\left(\tilde n(t+\tau,\cdot),\bar u_{Z(t)}\right)\leq \left(W_2\left(\tilde n(t,\cdot),\bar u_{Z(t)}\right)+\frac {C}\gamma\right)e^{-(\lambda_1-\lambda_2)\gamma\tau}.
\end{equation}
We notice next that thanks to Step~1, $|Z(t+\tau)-Z(t)|\leq C\tau$, which, combined to Theorem~\ref{thm:steady-states}, implies
\begin{equation}\label{eq:ttau}
 W_2\left(\bar u_{Z(t+\tau)},\bar u_{Z(t)}\right)\leq C\tau.
\end{equation}
Then,
\[W_2\left(\tilde n(t+\tau,\cdot),\bar u_{Z(t+\tau)}\right)\leq \left(W_2\left(\tilde n(t,\cdot),\bar u_{Z(t)}\right)+\frac {C}\gamma\right)e^{-(\lambda_1-\lambda_2)\gamma\tau}+C\tau.\]
If we choose $\tau:=\frac 1\gamma$, we get
\[W_2\left(\tilde n\left(t+1/\gamma,\cdot\right),\bar u_{Z(t+1/\gamma)}\right)\leq e^{-(\lambda_1-\lambda_2)} W_2\left(\tilde n(t,\cdot),\bar u_{Z(t)}\right)+ \frac {C}{\gamma},\]
where $e^{-(\lambda_1-\lambda_2)}<1$. As soon as 
\begin{equation}\label{eq:chose1}
W_2\left(\tilde n(t,\cdot),\bar u_{Z(t)}\right)\geq \frac 2 {1-e^{-(\lambda_1-\lambda_2)}}\frac {C}{\gamma}=\frac {C_1}\gamma,
\end{equation}
(we have
\[W_2\left(\tilde n\left(t+1/\gamma,\cdot\right),\bar u_{Z(t+1/\gamma)}\right)\leq \frac {1+e^{-(\lambda_1-\lambda_2)}}2 W_2\left(\tilde n(t,\cdot),\bar u_{Z(t)}\right).\]
As long as \eqref{eq:chose1} is satisfied for $W_2\left(\tilde n\left(t+k/\gamma,\cdot\right),\bar u_{Z(t+k/\gamma)}\right)$, we iterate this estimate, we get that for $k\in\mathbb N^*$,
\begin{eqnarray}
W_2\left(\tilde n\left(t+k/\gamma,\cdot\right),\bar u_{Z(t+k/\gamma)}\right)&\leq& \left(\frac {1+e^{-(\lambda_1-\lambda_2)}}2\right)^k W_2\left(\tilde n(t,\cdot),\bar u_{Z(t)}\right)\nonumber\\
&\leq&C\left(\frac {1+e^{-(\lambda_1-\lambda_2)}}2\right)^k,\nonumber
\end{eqnarray}
where we have used the estimate $W_2\left(\tilde n(t,\cdot),\bar u_{Z(t)}\right)\leq W_2\left(\tilde n(t,\cdot),\delta_0\right)+W_2\left(\delta_0,\bar u_{Z(t)}\right)\leq C$, thanks to \eqref{eq:2bound} and \eqref{eq:2bound-u}. Since $e^{-(\lambda_1-\lambda_2)}<1$, we can choose $\bar k$ such that $C\left(\frac {1+e^{-(\lambda_1-\lambda_2)}}2\right)^k< \frac {C_1}\gamma$, where $C_1$ is the constant defined in \eqref{eq:chose1}. There exists thus some $k\in\{0,\dots,\bar k\}$ such that \eqref{eq:chose1} does not hold for $W_2\left(\tilde n\left(t+k/\gamma,\cdot\right),\bar u_{Z(t+k/\gamma)}\right)$. Since this holds for any $t\geq 0$, we have shown that for any $t\geq \frac {\bar k}\gamma$,
\[\min_{s\in[t,t+\bar k/\gamma]}W_2\left(n\left(s,\cdot\right),\bar u_{Z(s)}\right)\leq \frac {C_1}\gamma.\]
This estimate combined to \eqref{eq:chose2} implies that there is a constant $C>0$ such that for all $t\geq \frac {\bar k}\gamma$, 
\begin{equation}\label{est:tilden-baru}
W_2\left(\tilde n\left(t,\cdot\right),\bar u_{Z(t)}\right)\leq \frac {C}{\gamma}.
\end{equation}

\medskip

\noindent\emph{Step 4: Estimates on $|Z(t)-\bar Z(t)|$}

Thanks to \eqref{eq:EDO} and \eqref{eq:Z2}, for $t\geq 0$,
\begin{equation}\label{eq:ZZ4}
|\bar Z(t)-Z_b|\leq |Z(0)-Z_b|e^{-rt},\quad |Z(t)-Z_b|\leq |Z(0)-Z_b|e^{-rt},
\end{equation}
and then 
\begin{equation}\label{eq:ZZ3}
|\bar Z-Z(t)|\leq \frac 1\gamma,
\end{equation}
for any $t\in\left[C\ln\gamma,\infty\right)$. We then simply need to estimate $|\bar Z-Z(t)|$ on $[0,C\ln\gamma]$. We estimate:
\begin{eqnarray}
Z(t)-\bar Z(t)&=&(Z(t)-Z_b)-(\bar Z(t)-Z_b)\nonumber\\
&=&(Z(0)-Z_b)e^{-\int_0^t \left(r+\int\alpha(x)\tilde n(s,x)\,dx\right)\,ds}-(Z(0)-Z_b)e^{-\int_0^t \left(r+\int\alpha(x)\bar u_{Z(s)}(x)\,dx\right)\,ds}\nonumber\\
&=&(Z(0)-Z_b)\left(e^{-\int_0^t\left( \int\alpha(x)(\tilde n(s,x)-\bar u_{\bar Z(s)}(x))\,dx\right)\,ds}-1\right)e^{-\int_0^t\left( r+\int\alpha(x)\bar u_{\bar Z(s)}(x)\,dx\right)\,ds}\label{eq:ZZ1}
\end{eqnarray}
To estimate the first term in this expression, we use the duality formula for $W_1$ (see \eqref{def:Wasserstein-dual-W1}):
\begin{eqnarray}
\left|\int\alpha(x)\left(\tilde n(s,x)-\bar u_{\bar Z(s)}(x)\right)\,dx\right|&\leq& \|\alpha'\|_\infty W_1\left(\tilde n(s,\cdot),\bar u_{\bar Z(s)}\right)\nonumber\\
&\leq& \|\alpha'\|_\infty W_2\left(\tilde n(s,\cdot),\bar u_{\bar Z(s)}\right)\nonumber\\
&\leq& \|\alpha'\|_\infty \left(W_2\left(\tilde n(s,\cdot),\bar u_{Z(s)}\right)+W_2\left(\bar u_{Z(s)},\bar u_{\bar Z(s)}\right)\right)\label{eq:ZZ5}
\end{eqnarray}
where the second inequality can be easily obtained from the definition~\ref{def:Wasserstein} of Wasserstein distances thanks to a Cauchy-Schwarz inequality. We can now apply the result of the previous step as well as Lemma~\ref{lem:contractionW2} to get
\begin{equation}\label{est:intalpha}
\left|\int\alpha(x)\left(\tilde n(s,x)-\bar u_{\bar Z(s)}(x)\right)\,dx\right|\leq \|\alpha'\|_\infty \left(\frac C\gamma+C|Z(s)-\bar Z(s)|\right),
\end{equation}
for any $t\geq \frac C\gamma$, while $\int_0^{C/\gamma}\left|\int\alpha(x)\left(\tilde n(s,x)-\bar u_{\bar Z(s)}(x)\right)\,dx\right|\,ds\leq \frac C\gamma$ and $\int_0^{C/\gamma}\left|Z(s)-\bar Z(s)\right|\,ds\leq \frac C\gamma$. \eqref{eq:ZZ1} then becomes
\begin{eqnarray*}
\left|Z(t)-\bar Z(t)\right|&\leq&(Z(0)-Z_b)\left(e^{C\left(\frac 1\gamma+\int_0^t\left(\frac 1\gamma+\left|Z(s)-\bar Z(s)\right|\right)\,ds\right)}-1\right)e^{-\int_0^t\left( r+\int\alpha(x)\bar u_{\bar Z(s)}(x)\,dx\right)\,ds}\\
&\leq& C\left(\frac {1+t}\gamma+\int_0^t\left|Z(s)-\bar Z(s)\right|\,ds\right)e^{-r t},
\end{eqnarray*}
As long as 
\begin{equation}\label{eq:condZZ}
\frac {1+t}\gamma+\int_0^t\left|Z(s)-\bar Z(s)\right|\,ds\leq 1.
\end{equation}
 $y(t)=e^{rt}\left|Z(t)-\bar Z(t)\right|$ then satisfies
\begin{eqnarray*}
y(t)&\leq& \frac {C(1+t)}\gamma+\int_0^tCe^{-rs}y(s)\,ds.
\end{eqnarray*}
We apply the Gronwall inequality to obtain
\[y(t)\leq \frac {C(1+t)}\gamma e^{\int_0^t e^{-rs}\,ds},%\leq \frac {C(t+1)}{\gamma r},
\]
and then 
\begin{equation}\label{eq:ZZ2}
\left|Z(t)-\bar Z(t)\right|\leq \frac {C(1+t)}{\gamma r} e^{-rt}.
\end{equation}
For $t\in [0,C\ln\gamma]$, we have then 
\begin{eqnarray*}
\frac {1+t}\gamma+\int_0^t\frac {C(1+s)}{\gamma r} e^{-rs}\,ds&\leq& \frac {1+C\ln\gamma}\gamma+\int_0^{C\ln\gamma}\frac {C(1+s)}{\gamma r} e^{-rs}\,ds\\
&\leq&\frac {1+C\ln\gamma}\gamma+\frac {C(1+s)}{\gamma r^2}<1,
\end{eqnarray*}
provided $\gamma>0$ is large enough. Then \eqref{eq:condZZ} is satisfied for all $t\in [0,C\ln \gamma]$, and \eqref{eq:ZZ2} holds for $t\in [0,C\ln \gamma]$, which, combined to \eqref{eq:ZZ3}, shows that if $\gamma>0$ is large enough, then for any $t\geq 0$, 
\begin{equation}\label{est:Z2}
\left|Z(t)-\bar Z(t)\right|\leq \frac C\gamma.
\end{equation}
If we combine\eqref{est:Z2} to \eqref{est:tilden-baru} and \eqref{eq:differents-bar-u0}, we can show that for all $t\geq \frac C\gamma$,
\begin{eqnarray}
W_2\left(\tilde n(t,\cdot),\bar u_{\bar Z(t)}\right)&\leq& W_2\left(\tilde n(t,\cdot),\bar u_{Z(t)}\right)+W_2\left(\bar u_{Z(t)},\bar u_{\bar Z(t)}\right)\nonumber\\
&\leq&\frac C\gamma+C|Z(t)-\bar Z(t)|\leq\frac C\gamma.\label{eq:estnu}
\end{eqnarray}

%\medskip
%

%
%\medskip
%
%\noindent\emph{Step 5: Estimates on $|N(t)-\bar N(t)|$ for $t\geq 0$ large}
%
%Thanks to \eqref{est:Z2}, there exists $\bar t\geq 0$ such that $\left|Z(t)-\bar Z(t)\right|\leq \varepsilon$ as soon as $t\geq\bar t$. 

\medskip

\noindent\emph{Step 4: Estimates on $|N(t)-\bar N(t)|$}

Let us first notice that $N'(t)\geq \left(r-\beta N(t)\right)N(t)$, and thus, for $t\geq 0$,
\begin{equation}\label{eq:minoration_N}
N(t)\geq \min\left(N(0),\frac r\beta\right).
\end{equation}
We define $\tilde N:=\frac 1\beta\left(r+\int\alpha(x) \bar u_{Z_b}(x)\,dx\right)$, and estimate:
\begin{eqnarray*}
\left(N-\tilde N\right)'(t)&=&\left(r+\int\alpha(x)\tilde n(t,x)\,dx-\beta N(t)\right) N(t)\\
&=&\left(\int\alpha(x)\left(\tilde n(t,x)-\bar u_{Z_b}(x)\right)\,dx-\beta \left(N(t)-\tilde N(t)\right)\right) N(t),
\end{eqnarray*}
To estimate the first term, we use an argument similar to the one used in \eqref{eq:ZZ5}, and then estimates \eqref{est:tilden-baru}, \eqref{est:Z2}, \eqref{eq:ZZ4} and Lemma~\ref{lem:contractionW2} to  shows that for $t\geq \frac C\gamma$,
\begin{align*}
&\left|\int\alpha(x)\left(\tilde n(t,x)-\bar u_{Z_b}(x)\right)\,dx\right|\\
&\quad \leq\|\alpha'\|_\infty\left(W_2\left(\tilde n(t,\cdot),\bar u_{Z(t)}\right)+W_2\left(\bar u_{Z(t)},\bar u_{\bar Z(t)}\right)+W_2\left(\bar u_{\bar Z(t)},\bar u_{Z_b}\right)\right)\\
&\quad \leq C\left(\frac 1\gamma+e^{-rt}\right),
\end{align*}
provided $\gamma>0$ is large enough. Then, for $t\geq \ln(\gamma)$ and $\gamma>0$ large enough, $\left|\int\alpha(x)\left(\tilde n(t,x)-\bar u_{Z_b}(x)\right)\,dx\right|\leq \frac C\gamma$, and then, thanks to \eqref{eq:minoration_N},
\begin{eqnarray*}
\frac d{dt}\left|N-\tilde N\right|(t)&\leq& \left(\frac C\gamma-\beta\left|N(t)-\tilde N(t)\right|\right)N(t)\\
&\leq& -\frac 1C\left|N(t)-\tilde N(t)\right|,
\end{eqnarray*}
provided $\left(N-\tilde N\right)(t)\geq \frac {C}{ \gamma}$ and $t\geq \ln(\gamma)$. Since moreover $\left(N-\tilde N\right)(\ln(\gamma))$ is bounded uniformly in $\gamma$ (see Step~1), we have, for $t\geq \ln(\gamma)$,
\begin{equation}\label{eq:estN-large-t}
\left|N-\tilde N\right|(t)\leq \max\left(Ce^{\frac {-t}C},\frac C\gamma\right),
\end{equation}
%\[\left|N-\tilde N\right|(t)\leq \frac C\gamma,\]
%for any $t\geq C\gamma\ln \gamma$. The same estimates can be made on $\bar N$, and thus, if $\gamma>0$ is large enough and $t\geq C\gamma\ln \gamma$,
%\begin{equation}\label{eq:estN-large-t}
%\left|N(t)-\bar N(t)\right|\leq \frac C\gamma.
%\end{equation}

Finally, we compute
\begin{eqnarray*}
(N-\bar N)'(t)&=&\left(r+\int \alpha(x)\tilde n(t,x)\,dx-\beta N(t)\right)\left(N(t)-\bar N(t)\right)\\
&&+\left(\int \alpha(x)\left(\bar u_{\bar Z(t)}(x)-\tilde n(t,x)\right)\,dx-\beta \left(\bar N(t)-N(t)\right)\right)\bar N(t),
\end{eqnarray*}
and then,
\begin{eqnarray*}
(N-\bar N)(t)&=&\int_0^t \left(\int \alpha(x)\left(\bar u_{\bar Z(s)}(x)-\tilde n(s,x)\right)\,dx\right)\bar N(s) \\
&&\quad \exp\left[\int_s^t\left(r+\int \alpha(x)\tilde n(\sigma,x)\,dx-\beta (\bar N(\sigma)+N(\sigma))\right)\,d\sigma\right]\,ds,
\end{eqnarray*}
%
%, we can estimate
%\begin{eqnarray*}
%\left|N(t)-\bar N(t)\right|&\leq& N(0)\,\Big| e^{\int_0^t \left(\int\alpha(x)\tilde n(s,x)\,dx-\beta N(s)\right)\,ds}\\
%&&\qquad -e^{\int_0^t \left(r+\int\alpha(x)\bar u_{\bar Z(s)}(x)\,dx-\beta \bar N(s)\right)\,ds}\Big|\\
%&\leq& N(0) \,\int_0^1 e^{\int_0^t \left[\int\alpha(x)\left(\lambda\tilde n(s,x)+(1-\lambda)\bar u_{\bar Z(s)}(x)\right)\,dx-\beta \left(\lambda N(s)+(1-\lambda)\bar N(s)\right)\right]\,ds}\,d\lambda\\
%&&\qquad \left|\int_0^{t}\left(\int\alpha(x)\left(\tilde n(s,x)-\bar u_{\bar Z(s)}(x)\right)\,dx-\beta \left(N(s)-\bar N(s)\right)\right)\,ds\right|.
%\end{eqnarray*}
Thanks to an argument similar to the one used to obtain \eqref{eq:ZZ5}, and then \eqref{est:tilden-baru}, \eqref{est:Z2} and Lemma~\ref{lem:contractionW2}, we can show that for $t\geq \frac C\gamma$,
\begin{align*}
&\left|\int\alpha(x)\left(\bar u_{\bar Z(s)}(x)-\tilde n(t,x)\right)\,dx\right|\\
&\quad \leq\|\alpha'\|_\infty\left(W_2\left(\tilde n(t,\cdot),\bar u_{Z(t)}\right)+W_2\left(\bar u_{Z(t)},\bar u_{\bar Z(t)}\right)\right)\leq \frac C\gamma,
\end{align*}
and since moreover $N$, $\bar N$ are bounded (see Step~1) and $\alpha$ is Lischitz continuous, we get
\[(N-\bar N)(t)\leq C\int_0^t\left(1_{s\in\left[0,C/\gamma\right]} +\frac 1\gamma\right) e^{C(t-s)}\,ds\leq \frac C\gamma e^{C_2 t},\]
and then, if we combine this estimate for $t\leq \frac {\ln\gamma}{2C_2}$ to \eqref{eq:estN-large-t} for $t\geq \frac {\ln\gamma}{2C_2}$, we get
\[\left|N(t)-\bar N(t)\right|\leq \max\left(Ce^{\frac {-ln\gamma }C},\frac C{\sqrt{\gamma}},\frac C\gamma\right),\]
that is 
\begin{equation}\label{est:N}
\left|N(t)-\bar N(t)\right|\leq C\gamma^{-\nu},
\end{equation}
for some $\nu>0$.
%we combine this to \eqref{eq:estN-large-t}, we get that for $t\geq 0$,
%\[\left|N(t)-\bar N(t)\right|\leq \min\left(Ce^{\frac {-t}C},\frac C\gamma e^{Ct}\right)+\frac C\gamma,\]
%and then, if we use 
%Then
%\begin{eqnarray*}
%\left|N(t)-\bar N(t)\right|&\leq& Ce^{Ct}\left(\int_0^t\left(1_{t\in[0,C/\gamma]}+\frac 1\gamma+\left|N(s)-\bar N(s)\right|\right)\,ds \right)\\
%&\leq& Ce^{Ct}\left(\frac {1+t}\gamma+\int_0^t\left|N(s)-\bar N(s)\right|\,ds \right).
%\end{eqnarray*}
%%we can then use \eqref{est:Z2}
%%\begin{equation*}
%%\left|N(t)-\bar N(t)\right|\leq C\left(\int_0^t\left(\frac 1\gamma+\frac {Cs}\gamma e^{-rs}+\left|N(s)-\bar N(s)\right|\right)\,ds \right)e^{Ct}.
%%\end{equation*}
%We introduce $y(t)=\left|N(t)-\bar N(t)\right|e^{-Ct}$, which satisfies
%
%
%PROBLEME $1/\gamma$ et $\varepsilon$.
%
%\newpage
%
%
%\begin{equation*}
%y(t)\leq \frac {C(1+t)}\gamma+ \int_0^tCy(s)e^{Cs}\,ds,
%\end{equation*}
%which implies, thanks to the Gronwall inequality, $y(t)\leq \frac {Ct}\gamma e^{\int_0^te^{Cs}\,ds}\leq \frac {Ct}\gamma e^{Ce^{Ct}}$, and then,
%\begin{equation}\label{est:N2}
%\left|N(t)-\bar N(t)\right|\leq \frac {Ct}\gamma e^{Ce^{Ct}+Ct},
%\end{equation}
%so that
%\begin{equation}\label{eq:estN-small-t}
%\forall t\leq \bar t,\quad \left|N(t)-\bar N(t)\right|\leq C\varepsilon,
%\end{equation}
%provided $\gamma>0$ is small enough.
Brought together, \eqref{est:Z2}, \eqref{est:N} and \eqref{eq:estnu} conclude the proof of Theorem~\ref{thm:hydro}.

\section{Convergence of the solution to a unique steady-state}
\label{sec:proof-thm2}

\subsection{Contraction of the exchanges in the $W_1-$distance}
\label{subsec:contractionW1}

\begin{lemma}\label{lem:contraction-1}
Let $x_1,\,x_2,\,x_1',\,x_2'\in\mathbb R_+^*$, and $K$ as in \eqref{def:K2}. We define $\lambda_1,\,\lambda_2\in [0,1]$ as in \eqref{def:lambdai}. We have
\begin{align}
&W_1\Big(K(\cdot,x_1,x_2),K(\cdot,x_1',x_2')\Big)\leq (1-\lambda_1)|x_1-x_1'|+\lambda_1|x_2-x_2'|.\label{eq:KW1}
\end{align}
\end{lemma}

\begin{proof}[Proof of Lemma~\ref{lem:contraction-1}]
The proof of the first part of Lemma~\ref{lem:contraction} can be reproduced for $(\varphi,\psi)\in\Phi_1$ until \eqref{eq:I=}. Similarly, we can use the fact that $(\varphi,\psi)\in\Phi_1$ (see \eqref{def:Phi}) to get
\begin{eqnarray*}
I&\leq&\int \left|(x_1-x_1y_1+x_2y_2)-(x_1'-x_1'y_1+x_2'y_2)\right|B\left(y_1\right) B\left(y_2\right)\,dy_1\,dy_2\\
&\leq& |x_1-x_1'|\int(1-y_1)B\left(y_1\right) \,dy_1+|x_2-x_2'|\int y_2 B\left(y_2\right)\,dy_2
\end{eqnarray*}
and then, with the notations \eqref{def:lambdai},
\begin{equation*}
I\leq |x_1-x_1'|(1-\lambda_1)+|x_2-x_2'|\lambda_1.
\end{equation*}
Since this is true for any $(\varphi,\psi)\in\Phi_1$ (see \eqref{def:Phi}), we can use this estimate and \eqref{def:Wasserstein-dual} to show that
\begin{eqnarray*}
W_1\Big(K(\cdot,x_1,x_2),K(\cdot,x_1',x_2')\Big)&=&\max_{(\varphi,\psi)\in \Phi_1}I\\
&\leq& |x_1-x_1'|(1-\lambda_1)+|x_2-x_2'|\lambda_1.
\end{eqnarray*}
\end{proof}

\subsection{Proof of Theorem~\ref{thm:hypo}}
\label{subsec:proof-thm-rep-cv}

\noindent\emph{Step 1: Rough estimate on the first moment of $\tilde n(t,\cdot)$}

We recall the notation $\tilde n(t,x)=\frac{n(t,x)}{N(t)}$, as well as the equation \eqref{eq:Z2} satisfied by $Z(t)=\int x\tilde n(t,x)\,dx$.
%
%We estimate the time evolution of $\bar x(t):=\int x\tilde n(t,x)\,dx$:
%\begin{eqnarray*}
%\frac d{dt}\bar x(t)&=&\left(r+\int \alpha (y) \tilde n(t,y)\,dy\right)\left(\int x n_b(x)\,dx-\bar x(t)\right)\nonumber\\
%&&+\gamma \left[\int\int \left(\int xK(x,x_1,x_2)\,dx\right)\tilde n(t,x_1) \tilde n(t,x_2)\,dx_1\,dx_2-\bar x(t)\right]\\
%&=&\left(r+\int \alpha (y) \tilde n(t,y)\,dy\right)\left(\int x n_b(x)\,dx-\bar x(t)\right),\nonumber\\
%&&+\gamma \left[\int\int \left(\int x\left(\frac{K(x,x_1,x_2)+K(x,x_2,x_1)}2\right)\,dx\right)\tilde n(t,x_1) \tilde n(t,x_2)\,dx_1\,dx_2-\bar x(t)\right]\\
%&=&\left(r+\int \alpha (y) \tilde n(t,y)\,dy\right)\left(\int x n_b(x)\,dx-\bar x(t)\right).
%\end{eqnarray*}
 For $t\geq 0$,
\begin{eqnarray*}
\left|Z(t)-\int x n_b(x)\,dx\right|&\leq& \left|Z(0)-\int x n_b(x)\,dx\right|\\
&\leq& \int x n(0,x)\,dx+\int x n_b(x)\,dx,
\end{eqnarray*}
and in particular,
\begin{eqnarray}
W_1(n_b,n(t,\cdot))&\leq&W_1(n_b,\delta_0)+W_1(\delta_0,n(t,\cdot))\nonumber\\
&\leq & W_1(n_b,\delta_0)+\int x\tilde n(t,x)\,dx\nonumber\\
&\leq&2\int x n(0,x)\,dx+\int x n_b(x)\,dx.\label{est:barx}
\end{eqnarray}

\medskip

\noindent\emph{Step 2: The contraction argument}

Let $n_1,\,n_2$ be two solutions of \eqref{eq:model-rep}, associated to two initial solutions $n_1^0,\,n_2^0$. We denote $\tilde n_1,\,\tilde n_2$ the two corresponding renormalized measures, and %define $Z_1(t),\,Z_2(t)$ their respective mean number of PgP, as well as \[\omega(t):=r+\alpha \frac{Z_1(t)+Z_2(t)}2.\]
\begin{equation}\label{def:omega}
\omega(t):=r+\frac 12\left(\int \alpha (y) \tilde n_1(t,y)\,dy+\int \alpha (y) \tilde n_2(t,y)\,dy\right).
\end{equation}
Thanks to \eqref{eq:model-tilde-n}, $\tilde n_i$, for $i=1,\,2$%, satisfies
%\begin{eqnarray*}
%\partial_t \tilde n_i(t,x)&=&\omega(t)\left(n_b(x)-\tilde n_i(t,x)\right)+ \frac 12\left(\int \alpha (y) \left(\tilde n_i(t,y)-\tilde n_{i^c}(t,y)\right)\,dy\right)\left(n_b(x)-\tilde n_i(t,x)\right)\nonumber\\
%&&+\gamma \left[\int\int K(x,x_1,x_2)\tilde n_i(t,x_1) \tilde n_i(t,x_2)\,dx_1\,dx_2-\tilde n_i(t,x)\right],
%\end{eqnarray*}
%where $i^c$ is defined by $i^c\in\{1,2\}\subset\{i\}$.
%
%For $i=1,\,2$, $\tilde n_i$ 
 can be written
\begin{eqnarray*}
\tilde n_i(t,x)&=&\tilde n_i^0(x)e^{-\int_0^t\omega(s)+\gamma\,ds}\\
&&+\int_0^t e^{-\int_s^t\omega(s)+\gamma\,ds}\Bigg[\left(r+\int \alpha (y) \tilde n_i(s,y)\,dy\right)n_b(x)\\
&&\quad+\left(\int \alpha (y) \frac{\tilde n_{i^c}(s,y)-\tilde n_i(s,y)}2\,dy\right)\tilde n_i(s,x)\\
&&\quad +\gamma\int\int K(x,x_1,x_2)\tilde n_i(s,x_1) \tilde n_i(s,x_2)\,dx_1\,dx_2\Bigg]\,ds,
\end{eqnarray*}
where $i^c=1$ if $i=2$, and $i^c=2$ if $i=1$. We consider $(\varphi,\psi)\in\Phi_1$, and estimate:
\begin{eqnarray}
I&=&\int \varphi(x)\tilde n_1(t,x)\,dx+\int \psi(X)\tilde n_2(t,X)\,dX\nonumber\\
&=&e^{-\int_0^t\omega(s)+\gamma\,ds}\left[\int \varphi(x)\tilde n_1^0(x)\,dx+\int \psi(X)\tilde n_1^0(X)\,dX\right]\nonumber\\
&&+\int_0^t e^{-\int_s^t\omega(s)+\gamma\,ds}\Bigg[\int \varphi(x)\Bigg(\left(r+\int \alpha (y) \tilde n_1(s,y)\,dy\right)n_b(x)\nonumber\\
&&\qquad+\left(\int \alpha (y) \frac{\tilde n_2(s,y)-\tilde n_1(s,y)}2\,dy\right)\tilde n_1(s,x)\Bigg)\,dx\nonumber\\
&&\qquad+\int \psi(X)\Bigg(\left(r+\int \alpha (y) \tilde n_2(s,y)\,dy\right)n_b(X)\nonumber\\
&&\qquad+\left(\int \alpha (y) \frac{\tilde n_1(s,y)-\tilde n_2(s,y)}2\,dy\right)\tilde n_2(s,X)\Bigg)\,dX\Bigg]\,ds\nonumber\\
&&+\gamma\int_0^t e^{-\int_s^t\omega(s)+\gamma\,ds}\Bigg[\int \varphi(x)\left(\int\int K(x,x_1,x_2)\tilde n_1(s,x_1) \tilde n_1(s,x_2)\,dx_1\,dx_2\right)\,dx\nonumber\\
&&\qquad+\int \psi(X)\left(\int\int K(x,x_1,x_2)\tilde n_2(s,x_1) \tilde n_2(s,x_2)\,dx_1\,dx_2\right)\,dX\Bigg]\,ds\label{eq:estIrep}
\end{eqnarray}
To estimate the first term of this expression, we simply use the formula \eqref{def:Wasserstein-dual} to obtain:
\begin{equation}\label{eq:estt0}
\int \varphi(x)\tilde n_1^0(x)\,dx+\int \psi(X)\tilde n_1^0(X)\,dX\leq W_1(\tilde n_1^0,\tilde n_2^0).
\end{equation}
Let 
\[\theta(t):= \int \alpha (y) \frac{\tilde n_1(t,y)-\tilde n_2(t,y)}2\,dy.\]
We estimate the second term of \eqref{eq:estIrep} as follows:
\begin{eqnarray*}
J&=&\omega(s)\left(\int \varphi(x)n_b(x)\,dx+\int \psi(X)n_b(X)\,dX\right)\nonumber\\
&&+\theta(s)\int (\varphi(x)-\psi(x))n_b(x)\,dx\nonumber\\
&&-\theta(s)\left(\int \varphi(x)\tilde n_1(s,x)\,dx-\int \psi(X)\tilde n_2(s,X)\,dX\right)\\
&=&\omega(s)\left(\int \varphi(x)n_b(x)\,dx+\int \psi(X)n_b(X)\,dX\right)\nonumber\\
&&+\theta(s)\left(\int \varphi(x)n_b(x)\,dx+\int \psi(x)\tilde n_2(X)\,dX\right)\nonumber\\
&&-\theta(s)\left(\int \varphi(x)\tilde n_1(s,x)\,dx+\int \psi(X)n_b(X)\,dX\right)
\end{eqnarray*}
%
%
%
%
%\begin{eqnarray*}
%J&=&\int \varphi(x)\left((r+\alpha Z_1(s))n_b(x)+\frac{Z_2(s)-Z_1(s)}2\tilde n_1(s,x)\right)\,dx\nonumber\\
%&&+\int \psi(X)\left((r+\alpha Z_2(s))n_b(X)+\frac{Z_1(s)-Z_2(s)}2\tilde n_2(s,X)\right)\,dX\nonumber\\
%&=&\int (\varphi(x)+\psi(x))\left(r+\alpha \frac{Z_1(s)+Z_2(s)}2\right)n_b(x)\,dx\\
%&&+\alpha\frac{Z_1(s)-Z_2(s)}2 \left(\int \varphi(x) n_b(x)\,dx+\int \psi(X) \tilde n_2(s,X)\,dX\right)\\
%&&-\alpha\frac{Z_1(s)-Z_2(s)}2 \left(\int \varphi(x) \tilde n_1(s,x)\,dx+\int \psi(X) n_b(X)\,dX\right).
%\end{eqnarray*}
Since $(\varphi,\psi)\in \Phi$, $\varphi(x)+\psi(x)=0$, and combining this property to the formula \eqref{def:Wasserstein-dual}, we get
\begin{equation*}
J\leq |\theta(s)|\Big(W_1(n_b,\tilde n_1(s,\cdot))+W_1(n_b,\tilde n_2(s,\cdot))\Big),
\end{equation*}
and then, thanks to \eqref{est:barx} and the definition of $\theta(\cdot)$ and \eqref{def:Wasserstein-dual},
\begin{equation}\label{est:J}
J\leq C_1\|\alpha'\|_\infty W_1\left(\tilde n_1(s,\cdot),\tilde n_2(s,\cdot)\right),
\end{equation}
where 
\begin{equation}\label{def:C1}
C_1:=4\int x n_b(x)\,dx+\int x (\tilde n_1(0,x)+\tilde n_2(0,x))\,dx.
\end{equation}
%
%
%assume w.l.o.g. that $\omega_1(t)\geq \omega_2(t)$, which implies that $\omega(t)=\omega_1(t)$. Then,
%\begin{align*}
%&\int \varphi(x)\left(\omega_1(s)n_b(x)+(\omega(s)-\omega_1(s))\tilde n_1(s,x)\right)\,dx+\int \psi(X)\left(\omega_2(s)n_b(x)+(\omega(s)-\omega_2(s))\tilde n_2(s,X)\right)\,dX\\
%&\quad =\int \varphi(x)\left(\omega_2(s)n_b(x)+(\omega_1(s)-\omega_2(s)) n_b\right)\,dx+\int \psi(X)\left(\omega_2(s)n_b(x)+(\omega_1(s)-\omega_2(s))\tilde n_2\right)\,dX\\
%&\quad\leq(\omega_1(s)-\omega_2(s))W_2^2(\tilde n_1(s,\cdot),\tilde n_b).
%\end{align*}
Finally, to estimate the last term of \eqref{eq:estIrep}, we reproduce the method employed in \eqref{est:L}, but with $(\varphi,\psi)\in \Phi_1$: for any $\pi_s\in \Pi\left(\tilde n_1(s,\cdot),\tilde n_2(s,\cdot)\right)$ (see \eqref{def:Wasserstein}),
\begin{eqnarray*}
L&=&\int \varphi(x)\left(\int\int K(x,x_1,x_2)\tilde n_1(s,x_1) \tilde n_1(s,x_2)\,dx_1\,dx_2\right)\,dx\nonumber\\
&&\qquad+\int \psi(X)\left(\int\int K(X,x_1,x_2)\tilde n_2(s,x_1) \tilde n_2(s,x_2)\,dx_1\,dx_2\right)\,dX\\
&\leq&\int\int \left(\int \varphi(x)K(x,x_1,x_2)\,dx+\int \psi(X)K(X,x_1',x_2')\,dX\right)\,d\pi_t(x_1,x_1')\,d\pi_t(x_2,x_2')\\
&\leq& \int\int W_1\left(K(\cdot,x_1,x_2),K(\cdot,x_1',x_2')\right)\,d\pi_t(x_1,x_1')\,d\pi_t(x_2,x_2')
\end{eqnarray*}
We can now use the second estimate of Lemma~\ref{lem:contraction} to get
\begin{eqnarray*}
L&\leq& \int\int (1-\lambda_1)|x_1-x_1'|+\lambda_1|x_2-x_2'| \,d\pi_t(x_1,x_1')\,d\pi_t(x_2,x_2')\\
&=&\int |x_1-x_1'|\,d\pi_t(x_1,x_1'),
%(1-\lambda_1)\int\left(\int (x_1-x_1')^2\,d\pi_s(x_1,x_1')\right)\,d\pi_s(x_2,x_2')\nonumber\\
%&&+\lambda_2\int\left(\int (x_2-x_2')^2\,d\pi_t(x_2,x_2')\right)\,d\pi_s(x_1,x_1')\nonumber\\
%&&+2(1-\lambda_1)\lambda_1\left(\int(x_1-x_1')\,d\pi_s(x_1,x_1')\right)\left(\int(x_2-x_2')\,d\pi_s(x_2,x_2')\right)
\end{eqnarray*}
and since this estimate holds for any $\pi_s\in \Pi\left(\tilde n_1(s,\cdot),\tilde n_2(s,\cdot)\right)$, we get (see \eqref{def:Wasserstein}):
\begin{equation}\label{est:L2}
L\leq W_1(\tilde n_1(s,\cdot),\tilde n_2(s,\cdot)).
\end{equation}
%the last term of this expression can be explicitly computed thanks to the marginals of $\pi_s$:
%\begin{align*}
%&\left(\int(x_1-x_1')\,d\pi_s(x_1,x_1')\right)\left(\int(x_2-x_2')\,d\pi_s(x_2,x_2')\right)\\
%&\quad=\left(\int x\tilde n_1(s,x)\,dx-\int x\tilde n_2(s,x)\,dx\right)^2=(Z_1(s)-Z_2(s))^2,
%\end{align*}
%while w

Finally, \eqref{eq:estIrep} becomes
\begin{eqnarray*}
I&\leq& e^{-\int_0^t\omega(s)+\gamma\,ds}W_1(\tilde n_1^0,\tilde n_2^0)\\
&&+\int_0^te^{-\int_s^t\omega(\sigma)+\gamma\,d\sigma}C_1\|\alpha'\|_\infty W_1\left(\tilde n_1(s,\cdot),\tilde n_2(s,\cdot)\right)\,ds\\
&&+\gamma \int_0^te^{-\int_s^t\omega(\sigma)+\gamma\,d\sigma}W_1(\tilde n_1(s,\cdot),\tilde n_2(s,\cdot))\,ds.
\end{eqnarray*}
Since this estimate is independent of $(\varphi,\psi)\in\Phi_1$, we can apply \eqref{def:Wasserstein-dual} to get
\begin{eqnarray*}
W_1\left(\tilde n_1(t,\cdot),\tilde n_2(t,\cdot)\right)&\leq& e^{-\int_0^t\omega(s)+\gamma\,ds}W_1(\tilde n_1^0,\tilde n_2^0)\\
&&+\int_0^te^{-\int_s^t\omega(\sigma)+\gamma\,d\sigma}\left(C_1\|\alpha'\|_\infty+\gamma\right) W_1\left(\tilde n_1(s,\cdot),\tilde n_2(s,\cdot)\right)\,ds.
\end{eqnarray*}

and then $y(t):=e^{\int_0^t\omega(s)+\gamma\,ds}W_1\left(\tilde n_1(t,\cdot),\tilde n_2(t,\cdot)\right)$ satisfies:
\begin{equation*}
y(t)\leq W_1(\tilde n_1^0,\tilde n_2^0)+\int_0^t\left(C_1\|\alpha'\|_\infty+\gamma\right)y(s)\,ds,
\end{equation*}
and then, thanks to the Gronwall inequality,
\begin{equation*}
y(t)\leq W_1(\tilde n_1^0,\tilde n_2^0)e^{\left(C_1\|\alpha'\|_\infty+\gamma\right)t},
\end{equation*}
that is
\begin{equation*}
W_1\left(\tilde n_1(t,\cdot),\tilde n_2(t,\cdot)\right)\leq  W_1(\tilde n_1^0,\tilde n_2^0)e^{C_1\|\alpha'\|_\infty t-\int_0^t\omega(s)\,ds}.
\end{equation*}
Thanks to the definition\eqref{def:omega} of $\omega$, $\omega\geq r+\min_{x\in \mathbb R_+}\alpha(x)$, and then,

\begin{equation}\label{est:n1n2}
W_1\left(\tilde n_1(t,\cdot),\tilde n_2(t,\cdot)\right)\leq  W_1(\tilde n_1^0,\tilde n_2^0)e^{-\left(r+\min_{x\in \mathbb R_+}\alpha(x)-C_1\|\alpha'\|_\infty \right)t}.
\end{equation}

\medskip

\noindent\emph{Step 4: Convergence of $\tilde n$}

To show that \eqref{eq:model-tilde-n} admits a steady-state, let $n(t,\cdot)\in \mathcal P_1(\mathbb R_+)$ be a solution of \eqref{eq:model-rep}, and $\tilde n$ the corresponding normalized solution. For any $0\leq \sigma\leq \tau$, we use the estimate \eqref{est:n1n2} with $\tilde n_1(t,x)=\tilde n(t,x)$, $\tilde n_2(t,x)=\tilde n((\tau-\sigma)+t,x)$, to get
\begin{eqnarray}
W_1\left(\tilde n_1(\sigma,\cdot),\tilde n_2(\tau,\cdot)\right)&\leq&  W_1(\tilde n(0,\cdot),\tilde n(\tau-\sigma,\cdot))e^{-\left(r+\min_{x\in \mathbb R_+}\alpha(x)-C_1\|\alpha'\|_\infty \right)\sigma}\nonumber\\
&\leq& \left(W_1(\tilde n(0,\cdot),\delta_0)+W_1(\delta_0,\tilde n(\tau-\sigma,\cdot))\right)e^{-\left(r+\min_{x\in \mathbb R_+}\alpha(x)-C_1\|\alpha'\|_\infty \right)\sigma}\nonumber\\
&\leq&C_1e^{-\left(r+\min_{x\in \mathbb R_+}\alpha(x)-C_1\|\alpha'\|_\infty \right)\sigma},\label{est:Cauchy1}
\end{eqnarray}
Thanks to the estimate~\eqref{est:barx} and the definition~\ref{def:C1} of $C_1$. It follows that for any sequence $t_n\to \infty$, $(\tilde n(t_n,\cdot))_n$ is a Cauchy sequence in the complete metric space $\left(\mathcal P_1(\mathbb R_+),W_1\right)$, and thus converge to a limit. Thanks to \eqref{est:Cauchy1}, this limit is indeed independent of the sequence $(t_n)$, so that there exists a limit $\bar n\in \mathcal P_1(\mathbb R_+)$ such that $\tilde n(t,\cdot)\to_{t\to\infty} \bar n(x)$ in $\left(\mathcal P_1(\mathbb R_+),W_1\right)$. It follows that $\tilde n(t,\cdot)$ is a steady-state of \eqref{eq:model-tilde-n}. 

Let now $n(t,x)$ be any solution of \eqref{eq:model-rep}, and $\tilde n$ defined as in  \eqref{def:tilden}. With the steady-state $\bar n$ of \eqref{eq:model-tilde-n} at hand, we can use \eqref{est:n1n2} with $\tilde n_1(t,x)=\bar n(x)$ and $\tilde n_2=\tilde n$, to show that $\tilde n$ converges exponentially fast to $\bar n$:
\begin{equation}\label{est:cv-tilde-n}
W_1\left(\tilde n(t,\cdot),\bar n\right)\leq  C_1\,e^{-\left(r+\min_{x\in \mathbb R_+}\alpha(x)-C_1\|\alpha'\|_\infty \right)t}.
\end{equation}

%This estimate shows that there exists a distribution $\bar n(x)\in \mathcal P_1(\mathbb R_+)$ such that for any solution $n$ of \eqref{eq:model-rep}, the corresponding distribution $t\mapsto n(t,\cdot)$ converges exponentially fast to $\bar n$, provided $r+\min_{x\in \mathbb R_+}\alpha(x)>C_1\|\alpha'\|_\infty$ (we recall the definition~\ref{def:C1} of $C_1$).

\medskip

\noindent\emph{Step 5: Convergence of $N$ and $n$}

Thanks to the definition \eqref{def:tilden} of $\tilde n$, we simply need to show that $N(t,x)=\int n(t,x)\,dx$ converges to the following limit:
\[\bar N:=\frac 1\beta\left(r+\int \alpha(x)\bar n(x)\,dx\right).\]
Thanks to \eqref{eq:N2}, $N$ satisfies:
\[N'(t)= \left(\int \alpha(x)\left(\bar n-\tilde n(t,x)\right)\,dx-\beta (N(t)-\bar N)\right)N(t),\]
where 
\begin{eqnarray*}
\left|\int \alpha(x)\left(\bar n-\tilde n(t,x)\right)\,dx\right|&\leq& \|\alpha'\|_\infty W_1\left(\bar n,n(t,\cdot)\right)\\
&\leq& C_1\|\alpha'\|_\infty e^{-\left(r+\min_{x\in \mathbb R_+}\alpha(x)-C_1\|\alpha'\|_\infty \right)t},
\end{eqnarray*}
thanks to \eqref{est:cv-tilde-n}. Thanks to this estimate, there exists $T>0$ such that for any $t\geq T$, $\left|\int \alpha(x)\left(\bar n-\tilde n(t,x)\right)\,dx\right|\leq \frac{\bar N}2$, and then, for $t\geq T$,
\[N'(t)\geq -\beta N(t)\left(N(t)-\frac{\bar N}2\right),\]
and then $N(t)\geq \frac{\bar N}4$ for any $t\geq T'$. Then, for $t\geq T'$,
\begin{eqnarray}
\left|N(t)-\bar N\right|&\leq&\left|N(T')-\bar N\right|e^{-\beta\int_{T'}^tN(s)\,ds}\nonumber\\
&&+\int_{T'}^t \left|\int \alpha(x)\left(\bar n-\tilde n(s,x)\right)\,dx\right|e^{-\beta\int_s^tN(\sigma)\,d\sigma}\,ds\nonumber\\
&\leq& \frac{3\bar N}4 e^{-\frac{\beta \bar N}4(t-T')}\nonumber\\
&&+C_1\|\alpha'\|_\infty \int_{T'}^t  e^{-\left(r+\min_{x\in \mathbb R_+}\alpha(x)-C_1\|\alpha'\|_\infty \right)s}e^{-\frac{\beta \bar N}4(t-s)}\,ds\nonumber\\
&\leq&C e^{-\min\left(r+\min_{x\in \mathbb R_+}\alpha(x)-C_1\|\alpha'\|_\infty ,\frac{\beta \bar N}4\right)t}.\label{est:cv-N}
\end{eqnarray}

Finally, we check that the convergence of $\tilde n$ and $N$ implies the convergence of $n$, for the weak-* topology of measures. Let $\varphi\in C^0(\mathbb R_+)$ and any $\varepsilon>0$. Thanks to the density of $C^1(\mathbb R_+)$ in $C^0(\mathbb R_+)$, there exists $\tilde \varphi\in C^1(\mathbb R_+)$ such that $\|\tilde \varphi'\|_\infty<\infty$, and $\|\varphi-\tilde \varphi\|_\infty\leq \varepsilon$. Then,
\begin{eqnarray*}
\left|\int \varphi(x) \left(n(t,x)-\bar N\bar n(x)\right)\,dx\right|&\leq&\|\varphi-\tilde\varphi\|_\infty\left(\|n(t,\cdot)\|_{L^1}+\bar N\|\bar n\|_{L^1}\right)\\
&&+\left|\int \tilde \varphi(x) \left(N(t,x)\tilde n(t,x)-\bar N \bar n(x)\right)\,dx\right|\\
&\leq&C\varepsilon+|N(t)-\bar N|\|\tilde\varphi\|_\infty+\bar N\|\tilde \phi'\|_\infty W_1(\tilde n(t,\cdot),\bar n)\\
&\leq& (C+1)\varepsilon,
\end{eqnarray*}
provided $t$ is large enough, thanks to  \eqref{est:cv-tilde-n} and \eqref{est:cv-N}. This estimates with \eqref{est:cv-tilde-n} and \eqref{est:cv-N} conclude the proof of Theorem~\ref{thm:hypo}.

\section*{Acknowledgements}
The second author acknowledges support from the ANR under grant  Kibord: ANR-13-BS01-0004, and MODEVOL: ANR-13-JS01-0009.


\begin{thebibliography}{9}


\bibitem{Abbot} P. Abbot et al., Inclusive fitness theory and eusociality. \emph{Nature} {\bf 471}.7339, E1--E4 (2011).

\bibitem{Bassetti2011}F. Bassetti, L. Ladelli, G. Toscani, Kinetic Models with Randomly Perturbed Binary Collisions, \emph{J. Stat. Phys.} {\bf 142}, 686--709.

\bibitem{Bisi2009} M. Bisi, G. Spiga, G. Toscani, Kinetic models of conservative economies with wealth redistribution.
\emph{Commun. Math. Sci.} {\bf7}, 901--916 (2009)


\bibitem{Bisi}M. Bisi, J. A. Carrillo, G. Toscani, Contractive Metrics for a Boltzmann equation for granular gases: Diffusive equilibria. \emph{J. Stat. Phys.}, {\bf 118}(1-2), 301--331 (2005).



\bibitem{Bobylev1} A.V. Bobylev, Exact solutions of the Boltzmann equation, \emph{Doklady Akad. Nauk SSSR} {\bf 231}, 571--574 (1975).

\bibitem{Bobylev2} A.V. Bobylev, The theory of the nonlinear spatially uniform Boltzmann equation for Maxwell molecules, {\bf Sov. Sci. Rev. C. Math. Phys.} {\bf 7}, 111--233 (1988).

\bibitem{Bolley}F. Bolley, J. A. Carrillo, Tanaka theorem for inelastic Maxwell models. \emph{Commun. Math. Phys.} {\bf 276}(2) 287--314 (2007).

\bibitem{Bulmer} B\"ulmer, M.G., 1980. The Mathematical Theory of Quantitative Genetics. Clarendon Press, Oxford, UK.

\bibitem{Burger} Burger, R., 2000. The Mathematical Theory of Selection, Recombination and Mutation. Wiley, New-York.

\bibitem{Carrillo-rev}J. A. Carrillo,  G. Toscani,  Contractive probability metrics ans asymptotic behavior of dissipative kinetic equations. \emph{Riv. Mat. Univ. Parma}, {\bf 7}(6), 75--198 (2007).

\bibitem{Davis} D. M. Davis, Intercellular transfer of cell-surface proteins is common and can affect many stages of an immune response. \emph{Nat. Rev. Immunol.} {\bf 7}(3), 238--243 (2007).


\bibitem{Hinow} P. Hinow, F. Le Foll, P. Magal, G. F. Webb, Analysis of a
model for transfer phenomena in biological populations. \emph{SIAM J. Appl. Math.} {\bf 70}(1), 40--62 (2009).

\bibitem{Levchenko}A. Levchenko, B. M. Mehta, X. Niu, G. Kang, L. Villafania, D. Way, D. Polycarpe, M. Sadelain, S. M. Larson, Intercellular transfer
of P-glycoprotein mediates acquired multidrug resistance in tumor cells.
\emph{Proc. Natl. Acad. Sci. U.S.A.} {\bf 102}, 1933--1938 (2005).

\bibitem{Marshall} J. A. Marshall, Group selection and kin selection: formally equivalent approaches. \emph{Trends Ecol. Evol.}, {\bf 26}(7), 325--332 (2011).

\bibitem{Mattes} D. Matthes, G. Toscani,  On steady distributions of kinetic models of conservative economies. \emph{J. Stat.
Phys.} {\bf 130}, 1087--1117 (2008).

\bibitem{Mirrahimi} S. Mirrahimi, G. Raoul, Population structured by a space variable and a phenotypical trait, \emph{Theor. Popul. Biol.} {\bf 84}, 87--103 (2013).

\bibitem{Nieth}C. Nieth, H. Lage, Induction of the ABC-transporters Mdr1/
P-gp (Abcb1), mrpl (Abcc1), and bcrp (Abcg2) during establishment of
multidrug resistance after exposure to mitoxantrone. \emph{J. Chemother.} {\bf 17}, 215--223 (2005).

\bibitem{Nowak} M. A. Nowak et al., The evolution of eusociality. \emph{Nature} {\bf 466}, 1057--1062 (2010).

\bibitem{Pasquier}J. Pasquier, L. Galas, C. Boulangé-Lecomte, D. Rioult, F. Bultelle, P. Magal, G. Webb, F. Le Foll, Different modalities of intercellular membrane exchanges mediate cell-to-cell p-glycoprotein transfers in MCF-7 breast cancer cells. \emph{J. Biol. Chem.,} {\bf 287}(10), 7374--7387 (2012).

\bibitem{Pasquier2} J. Pasquier, P. Magal, C. Boulangé-Lecomte, G. Webb, F. Le Foll, Consequences of cell-to-cell P-glycoprotein transfer on acquired multidrug resistance in breast cancer: a cell population dynamics model. \emph{Biol. Direct}, {\bf 6}(5) (2011).

\bibitem{Pastan} I. Pastan, M. Gottesman, Multiple-drug resistance in human
cancer. \emph{N. Engl. J. Med.} {\bf 316}, 1388--1393 (1987).

\bibitem{Tanaka} H. Tanaka, Probabilistic treatment of the Boltzmann equation of Maxwellian molecules. \emph{Z. Wahrsch. Verw. Gebiete} {\bf 46}(1), 67--105 (1978/79).

\bibitem{Thery} C. Th\'ery, M. Ostrowski, E. Segura, Membrane vesicles as conveyors of immune responses. \emph{Nat. Rev. Immunol.}, {\bf 9}(8), 581--593 (2009).

\bibitem{Villani} C. Villani, Optimal transport: old and new. Vol. 338. Springer Science \& Business Media, 2008.

\bibitem{West}S. A. West, S. P. Diggle, A. Buckling, A. Gardner, A. S. Griffin, The social lives of microbes. \emph{Annu. Rev. Ecol. Evol. Syst.} {\bf 38}, 53--77 (2007).

\end{thebibliography}
\end{document}